\documentclass{amsart}

\usepackage{amsmath,amsfonts,amsthm,amssymb}
\usepackage[alphabetic]{amsrefs}
\usepackage[OT4]{fontenc}
\usepackage{float}
\usepackage{enumerate}
\usepackage[utf8]{inputenc}
\usepackage[polish,english]{babel}
\usepackage{tikz}
\usepackage{xcolor}
\usetikzlibrary{shapes.multipart,shapes.arrows}
\usetikzlibrary{patterns,arrows,decorations.pathreplacing}

\newtheorem{theorem}{Theorem}[section]

\newtheorem{lemma}[theorem]{Lemma}
\newtheorem{lem}[theorem]{Lemma}
\newtheorem{theo}[theorem]{Theorem}

\newtheorem{prop}[theorem]{Proposition}

\newtheorem{cor}[theorem]{Corollary}

\theoremstyle{definition}
\newtheorem{rem}[theorem]{Remark}

\newtheorem{defi}[theorem]{Definition}

\newcommand{\parti}{\mathrm{\tilde{\partial}}}

\newcommand{\cay}{\widetilde{X}}

\theoremstyle{definition}
\newtheorem{definition}[theorem]{Definition}

\theoremstyle{remark}

\numberwithin{equation}{section}

\title{Cubulating random groups in the square model}
\address{Institute of Mathematics, Polish Academy of Sciences,
Warsaw, {\'S}niadeckich 8}
\email{tomaszo@impan.pl}
\author{Tomasz Odrzyg{\'o}{\'z}d{\'z}}

\begin{document}

\begin{abstract}
%In \cite{odrz} we introduced the square model for random groups, which is a modification of the  \.Zuk's triangular model: in the square model we draw at random a presentation consisting of relators of length four. In this paper we present futher research on this model. 
Our main result is that for densities $<\frac{3}{10}$ a random group in the square model has the Haagerup property and is residually finite. Moreover, we generalize the Isoperimetric Inequality, to some class of non-planar diagrams and,  using this, we introduce a system of modified hypergraphs  providing the structure of a space with walls on the Cayley complex of a random group. Then we show that the natural action of a random group on this space with walls is proper, which gives the proper action of a random group on a CAT(0) cube complex. %The idea of correcting  hypergraphs was inspired by \cite{przyt}.
\end{abstract}

\maketitle

\section{Introduction}

%In \cite{gro93} Gromov introduced  the notion of a random finitely presented group on $m \geq 2$ generators at density $d \in (0,1)$. The idea was to fix a set of $m$ generators and consider presentations with $(2m-1)^{dl}$ relators chosen uniformly, each of which is a random cyclically reduced word of length $l$. Gromov investigated the properties of random groups when $l$ goes to infinity. We say that a property occurs in the Gromov density model \textit{with overwhelming probability (w.o.p.)} if the probability that a random group has this property converges to $1$ when $l \rightarrow \infty$.

%There are many important properties of this model: for densities $> \frac{1}{2}$ a random group is trivial w.o.p. (see \cite[Theorem 1]{shape} and \cite{gro93}), for densities $< \frac{1}{2}$ a random group is, w.o.p., infinite, hyperbolic and torsion-free (see \cite[Theorem 1]{shape} and \cite{gro93}) and for densities $>\frac{1}{3}$ a random group w.o.p. has Property (T) (see \cite{kot}).

In \cite{odrz} we introduced the square model  for random groups, where we draw at random relations of length four. The motivation was that the Cayley complex of such a group has a natural structure of a square complex, which should be easier to analyze than the polygonal Cayley complexes for groups in the Gromov model.

\begin{defi}\label{Definition:Square_model}[Square model, {\cite[Definition 1.3]{odrz}}]
Consider the set $A_n = \{a_1, \dots, a_n \}$, which we will refer to as an alphabet. Let $W_n$ be the set of all cyclically reduced words of length 4 over  $A_n$. Note that $|W_n| = (2n-1)^4$ up to a multiplicative constant. By $\mathbb{F}_n$ we will denote the free group generated by the elements of $A_n$. By \textit{relators} we will understand words over generators and by \textit{relations} the corresponding equalities holding in the group.

For $d \in (0,1)$ let us choose randomly, with uniform distribution, a subset $R_n \subset W_n$ such that $|R_n| = \lfloor (2n-1)^{4d} \rfloor$. Quotienting $\mathbb{F}_n$ by the normal closure of the set $R_n$, we obtain a \textit{random group in the square model at density d}.

We say that property $\mathcal{P}$ occurs in the square model at density $d$ \textit{with overwhelming probability} (w.o.p.) if the probability that a random group has property $\mathcal{P}$ converges to $1$ as $n \rightarrow \infty$.
\end{defi}
In \cite{odrz} we showed that a random group in the square model at density $d$ w.o.p.:
\begin{itemize}
\item is trivial for $d > \frac{1}{2}$,
\item is infinite, hyperbolic, torsion-free of geometric dimension 2 for $d < \frac{1}{2}$,
\item is free for $d < \frac{1}{4}$
\item does not have Property (T) for  $d < \frac{1}{3}$
\end{itemize}

In this paper we prove

\begin{theo}\label{Theorem:Cubulating}
For densities $d<\frac{3}{10}$ a random group in the square model w.o.p. acts properly and cocompactly on a CAT(0) cube complex.
\end{theo}

\begin{cor}\label{Corollary:Cubulating}
For densities $d<\frac{3}{10}$ a random groups in the square model w.o.p. has the Haagerup property and is residually finite.
\end{cor}

\begin{proof}
The residual finiteness results from \cite[Corollary 1.2]{agol} and the Haagerup Property by a folklore remark (see for example  \cite{cmv}).
\end{proof}

To obtain Theorem \ref{Theorem:Cubulating} we needed to generalize the \textit{isoperimetric inequality} (\cite[Theorem 2]{some}) to a class of non-planar diagrams (see Theorem \ref{Theorem:Generalized_Isoperimetric_Inequality}). This generalization has already proved to be useful in other work on random groups (see for example \cite{przyt}). Theorem \ref{Theorem:Generalized_Isoperimetric_Inequality} is stated in the regime of the square model, but the reasoning can be easily repeated in the Gromov model (see \cite{notka} for some version of it). The main idea of the proof of Theorem \ref{Theorem:Cubulating} is based on introducing modified hypergraps, which provide a structure of a space with walls on the Cayley complex of a random group, that has desired metric properties. This modified walls allowed us to construct a proper action of a random group on a space with walls, which was impossible to obtain using standard hypergraphs (introduced in \cite{ow11} for the Gromov model and in \cite{odrz} for the square model).

\subsection*{Organization}The paper is organized as follows: first we prove a generalization of the isoperimetric inequality, then using this we introduce ``corrected'' hypergraphs that provide the structure of a space with walls on the Cayley complex of a random group. Finally we show that a random group acts properly on this space with walls. The idea of correcting
 hypergraphs was inspired by \cite{przyt}.

\subsection*{Acknowledgements}The author would like to thank Piotr Przytycki for many discussions and Piotr Nowak for general advice on writing.

\section{Non-planar isoperimetric inequality}

The goal of this section is to generalize the isoperimetric inequality to some class of
non-planar complexes, but having ``disc-like” structure, meaning that they are unions of a large disc component and some bounded non-planar pieces. Firstly, let us recall the isoperimetric inequality:

\begin{theo}[{\cite[Theorem 2]{some}}]\label{Theorem:Isoperimetric_inequality} For any $\varepsilon > 0$, in the Gromov model at density $d < \frac{1}{2}$, with overwhelming probability all reduced van Kampen diagrams associated to the group presentation satisfy

\begin{equation} \label{Equation:Isoperimetric_inequality}
|\partial D| \geq l(1 - 2d -  \varepsilon) |D|.
\end{equation}
Here $\partial D$ denotes the set of boundary edges of the diagram $D$ and $|D|$ denotes the number of 2-cells of $D$.
\end{theo}
One of the corollaries of Theorem \ref{Theorem:Isoperimetric_inequality} is that in the Gromov density model at densities $<\frac{1}{2}$ a random group is w.o.p. hyperbolic (see \cite[Theorem 1]{shape} and \cite{gro93}). In \cite[Theorem 3.15]{odrz} the author proved that Equation (\ref{Equation:Isoperimetric_inequality}) is satisfied w.o.p. as well for the square model of random groups.

Let us now introduce some definitions.

\begin{defi}
Suppose $Y$ is a finite $2$--complex, not necessarily a disc diagram.
\begin{itemize}
\item The \textit{generalized boundary length} of $Y$, denoted $|\tilde{\partial} Y|$, is $$ |\tilde{\partial} Y| := \sum_{e\in Y^{(1)}} (2 - \deg(e)),$$
\item and the \textit{cancellation} in $Y$ is
 $$\mathrm{Cancel}(Y) := \sum_{e\in Y^{(1)}} (\deg(e)-1),$$
\end{itemize}
where $\mathrm{deg(e)}$ is the number of times that $e$ appears as the image of an edge of the attaching map of a $2$–cell of $Y$.
\end{defi}

Note, that for every planar diagram $Y$ we have $|\tilde{\partial} Y| = |\partial Y|$. Moreover, $\mathrm{Cancel}(Y)=\frac{1}{2}(4|Y|-|\partial Y|)$. %The notation of generalized boundary length was mentioned to be consistent with the notation of the standard boundary.

\begin{defi}
We say that a finite 2-complex $Y$ is \emph{fulfilled} by a set of relators $R$ if there is a combinatorial map from~$Y$ to the presentation complex $\cay /G$ that is locally injective around edges (but not necessarily around vertices).
\end{defi}

In particular, any subcomplex of the Cayley complex $\cay$ is fulfilled by~$R$.

\begin{defi}[Diagram with $K$-small legs]\label{Definition:Diagram_with_small_legs}
Let $G = \left< S|R \right>$ be a finite group presentation where $R$ consists of cyclically reduced words of length 4 over  $S$. For $K > 0$ let $Y$ be a
2–complex that is a union of a reduced Van Kampen diagram $Z$ and a family $Z_i$ of connected complexes, called \textit{legs of $Y$}, such that $|Z_i| \leq K$, for each $i$ the leg $Z_i$ contains an external vertex of $Z$ and that every edge of $Y$ belongs to maximally two legs. We call such $Y$ a \textit{diagram with $K$-small legs (with
respect to the presentation $\left< S|R \right>$)}. The disc diagram $Z$ is called the \textit{disc basis of $Y$}.
%The complexes $Z_i$ we call \textit{legs of $Y$}.
\end{defi}

Until the end of this section $G$ will denote the random group in the square model with the presentation $\left<S | R \right>$  and $\cay$ --- the Cayley complex of $G$ with respect to this presentation.  %Let $m$ denote the cardinality of the set $S$.

\begin{theo}[Generalized Isoperimetric Inequality]\label{Theorem:Generalized_Isoperimetric_Inequality}
In the square model at density $d \leq \frac{1}{2}$ the following statement holds w.o.p.: for each $K$ and $\varepsilon > 0$ there is no diagram $Y$ with $K$-small legs fullfilable by $R$ and satisfying:

\begin{equation}
\emph{Cancel}(Y) > 4(d + \varepsilon)|Y|,
\end{equation}
or equivalently
\begin{equation}
|\tilde{\partial} Y| > 4(1 - 2d - \varepsilon)|Y|.
\end{equation}
\end{theo}

Our strategy to prove Theorem \ref{Theorem:Generalized_Isoperimetric_Inequality} is to first prove a ``local'' version of it, that is with the additional limit on the number of 2-cells in a diagram, and then to show that this locality assmuption can be ommited. Analogous results for the Gromov model  was obtained by the author in \cite{notka} and used in \cite{przyt} to construct balanced walls. We thank Piotr Przytycki for suggesting the statement of Theorem \ref{Theorem:Generalized_Isoperimetric_Inequality} involving the notion of cancellation.

\subsection{Local version of the generalized isoperimetric inequality}

\begin{lem}\label{Lemma:Local_generalized_isoperimetric}
In the square model at density $d \leq \frac{1}{2}$, for each $K, \varepsilon >0$ w.o.p. there is no 2-complex $Y$ with $|Y| \leq K$ fulfilled by $R$ and satisfying
\begin{equation}
\emph{Cancel}(Y) > 4(d + \varepsilon)|Y|.
\end{equation}
\end{lem}

Our proof of Lemma \ref{Lemma:Local_generalized_isoperimetric} will be only a slight modification of Olliver's proof of Theorem \ref{Theorem:Isoperimetric_inequality}. We postpone it to to introduce some necessary tools first.

\begin{defi}\label{Definition:Abstract_2_complex}
Let $Y$ be a $2$-complex. Suppose that the following additional information is given:
\begin{enumerate}
\item Every face of $Y$ is labeled by a number $i \in \{1, 2, \dots, n \}$. The labels can repeat and for each number $1 \leq i \leq n$ there is at least one face labeled with $i$. We call $n$ the \textit{number of distinct relators}.
\item For each face there is a distinguished edge, called \textit{the starting edge} and an orientation.
\end{enumerate}
Such $Y$ will be referred to as an \textit{abstract $2$-complex}. We say that $Y$ is \textit{fulfilled} by a sequence of relators $(r_1, r_2, \dots, r_n )$ if it is fulfilled by $\{r_1, r_2, \dots, r_n \}$ as $2$-complex, in such a way that each face $f$ labelled by $i \in \{1,2,\ldots,n\}$ (according to (1) in Definition \ref{Definition:Abstract_2_complex}) is sent under the combinatorial map to the face of $X/G$ corresponding to $r_i$, so that the image in $X^{(1)}/G$ of the attaching path of $f$, starting at the oriented starting edge (according to $(2)$ in Definition \ref{Definition:Abstract_2_complex}) reads off the word $r_i$. If $f$ is labeled by $i$ we say that $f$ bears the relator $i$.
 \end{defi}

Recall that $S$ denotes the cardinality of the generating set of a random group in the square model. Let $m$ denote the cardinality of the set $S$.

\begin{prop}\label{Proposition:Fullfill}
Let $R$ be a random set of relators at density $d$ and at length $4$ on $m$ generators. Let $Y$ be an abstract $2$-complex. Then either $\mathrm{Cancel}(Y) < 4(d+2\varepsilon)|Y|$ or the probability that there exists a tuple of relators in $R$ fulfilling $Y$ is less than $(2m-1)^{- 4\varepsilon}$.
\end{prop}

Before we prove this proposition we introduce some additional notation. Let $n$ be the number of distinct relators in an abstract 2-complex $Y$. For $1 \leq i \leq n$ let $m_i$ be the number of times relator $i$ appears in $Y$. Up to reordering the relators we can suppose that $m_1 \geq m_2 \geq \dots \geq m_n$.

For $1 \leq i_1, i_2 \leq n$ and $1 \leq k_1, k_2 \leq l$ we say that $(i_1, k_1) > (i_2, k_2)$ if $i_1 > i_2$ or $i_1 = i_2$ but $k_1 > k_2$ (lexicographic order). Suppose that for some $s \geq 2$ an edge $e$ of $Y$ is adjacent to faces $f_1, f_2, \dots, f_s$ labeled by $i_1, i_2, \dots, i_s$ accordingly. Suppose moreover that for $1 \leq j \leq s$ the edge $e$ is the $k_j$--th edge of the face $f_j$. Since $Y \rightarrow X/G$ is locally injective around $e$, the pairs $(i_j,k_j)$ are distinct. Choose $j=j_{min}$ for which $(i_j,k_j)$ is minimal. We say that edge $e$ \textit{belongs to faces} $f_j$ for $j \in \{1, 2, \dots s\} \setminus \{j_{min}\}$.

Let $\delta(f)$ be the number of edges belonging to a face $f$. For $1 \leq i \leq n$ let

$$\kappa_i = \max \{ \delta(f) : f \text{ is a face labeled by relator i} \}$$

Note that

\begin{equation}\label{eq:2}
\mathrm{Cancel}(Y) = \sum_{f \in Y^{(2)}} \delta(f) \leq  \sum_{1 \leq i \leq n} m_i \kappa_i
\end{equation}

\begin{defi}
Let $Y$ be an abstract $2$-complex with $n$ distinct relators. For $1 \leq k \leq n$ let $R' = (w_1, w_2, \dots, w_k)$ be a sequence of relators.  We say that $Y$ is \textit{partially fulfilled by} $R'$ if the abstract $2$-complex $Y' \subset Y$ that is the closure of the faces of $Y$ labeled by the numbers $i \in \{1,2,...,k\}$ is fulfilled by $R'$.
\end{defi}

\begin{lem}\label{lem:59}
For $1 \leq i \leq n$ let $p_i$ be the probability that $i$ randomly chosen words $w_1, w_2, \dots, w_i$ partially fulfill $Y$ and let $p_0 = 1$. Then
\begin{equation}\label{eq:3}
\frac{p_i}{p_{i-1}} \leq (2m-1)^{-\kappa_i}.
\end{equation}
\end{lem}

\begin{proof}
Suppose that first $i-1$ words $w_1, \dots, w_{i-1}$ partially fulfilling $Y$ are given. We will successively analyze what is the choice for the consecutive letters of the word $w_i$, so that $Y$ is fulfilled. Let $k \leq 4$ and suppose that the first $k-1$ letters of $w_i$ are chosen. Let $f$ be the face realizing $\delta(f)=\kappa_i$ and let $e$ be the $k$-th edge of the face $f$.

If $e$ belongs to $f$ this means that there is another face $f'$ meeting $e$ which bears relator $i' < i$ or bears $i$ too, but $e$ appears in $f'$ as a $k' < k$-th edge. In both cases the letter on the edge $e$ is imposed by some letter already chosen so drawing it at random has probability $\leq \frac{1}{(2m-1)}$ up to some small error. In fact, this estimate would be valid only if the words were only required to be reduced; since they are cyclically reduced a negligible error appears, which we ignore.

Combining all these observations we get that the probability to choose at random the correct word $w_i$ is at most $p_{i-1}(2m-1)^{-\kappa_i}$.
\end{proof}

Now we can provide the proof of Proposition \ref{Proposition:Fullfill}

\begin{proof}[Proof of Proposition \ref{Proposition:Fullfill}]
For $1 \leq i \leq n$ let $P_i$ be the probability that there exists an $i$-tuple of words partially fulfilling $Y$ in the random set of relators $R$. We trivially have:

\begin{equation}\label{eq:4}
P_i \leq |R|^i p_i = (2m-1)^{4id}p_i
\end{equation}

Combining equations (\ref{eq:2}) and (\ref{eq:3}) we get

$$\mathrm{Cancel}(Y) \leq  \sum_{i=1}^n m_i(\log_{2m-1}p_{i-1} - \log_{2m-1}p_i)=$$

$$= \sum_{i=1}^{n-1} (m_{i+1} - m_i)\log_{2m-1}p_i -  m_n \log_{2m-1}p_n +  m_1 \log_{2m-1}p_0. $$
Now $p_0 = 1$ so $\log_{2m-1}p_0 = 0$ and we have
$$\mathrm{Cancel}(Y) \leq  \sum_{i=1}^{n-1}(m_{i+1} - m_i)\log_{2m-1}p_i -  m_n \log_{2m-1}p_n.$$
From (\ref{eq:4}) and the fact that $m_{i+1} - m_i \leq 0$ we obtain
$$\mathrm{Cancel}(Y) \leq \sum_{i=1}^{n-1} (m_{i+1} - m_i)(\log_{2m-1}P_i - 4 i d) -  m_n \log_{2m-1}(P_n - 4 n d)$$
Observe that $\sum_{i=1}^{n-1} (m_i - m_{i+1}) i + m_n n  =  \sum_{i=1}^n m_i = |Y|$. Hence
$$\mathrm{Cancel}(Y) \leq 4|Y| d +  \sum_{i=1}^{n-1} (m_{i+1} - m_i)\log_{2m-1}P_i -  m_n \log_{2m-1}P_n$$
Setting $P = \min_i P_i$ and using the fact that $m_{i+1} - m_i \leq 0$ we get

$$\mathrm{Cancel}(Y) \leq 4|Y|d + (\log_{2m-1}P) \sum_{i=1}^{n-1} (m_{i+1} - m_i) -  m_n \log_{2m-1}P =$$
$$= 4|Y|d - m_1 \log_{2m-1}P \leq |Y|(4 d - \log_{2m-1}P),$$
since $m_1 \leq |Y|$. It is clear that a complex is fulfillable if it is partially fulfillable for any $i \leq n$ and so:

$$\mathrm{Probability(}\emph{Y is fullfillable by relators of R}\mathrm{)} \leq P \leq (2m-1)^{\frac{4|Y| d - \mathrm{Cancel}(Y)}{|Y|}},$$
which was to be proven.
\end{proof}

\begin{proof}[Proof of Lemma \ref{Lemma:Local_generalized_isoperimetric}]
Denote by $C(K,m)$ the number of abstract square complexes with at most $K$ 2-cells. Observe that there are finitely many square complexes with at most $K$ faces. There are also finitely many ways to decide which faces would bear the same relator, and also finitely many ways to choose the starting point of each relator. Therefore, the values $\{C(K, m)\}_{m \in \mathbb{N}}$ with fixed $K$ have a uniform bound $M$. We know by Proposition \ref{Proposition:Fullfill} that for any abstract $2$-complex with at most $K$ faces violating the inequality $\mathrm{Cancel}(Y) < 4(d+\varepsilon)|Y|$ the probability that it is fulfilled by a random set of relators is $\leq (2m-1)^{- 4\varepsilon}$. So the probability that there exists a fulfilled $2$-complex with at most $K$ faces, violating the inequality is $\leq C(K, m)(2m-1)^{- 4\varepsilon l} \leq M (2m-1)^{- 4\varepsilon l}$, so it converges to $0$ as $m \rightarrow \infty$.
\end{proof}

\subsection{From the local version to the global}

We start this subsection by reformulating \cite[Lemma 11]{some}  by replacing the length of relator $l$ by 4.

% ``class P'' by the ``class of van Kampen diagrams'' to obtain:

\begin{lemma}\label{cut}
Let $G = \left<S | R \right>$ be a finite presentation in which all elements of $R$ have length 4. Suppose that for some constant $C'>0$ every van Kampen diagram $E$ of this presentation satisfies:

$$ |\partial E| \geq 4 C' |E|.$$
Then every van Kampen diagram $D$ can be partitioned into two diagrams $D'$, $D''$ by cutting it along a path of length at most $4 + 8\frac{\log(|D|)}{C'}$ with endpoints on the boundary of $D$ such that each of $D'$ and $D''$ contains at least one quarter of the boundary of $D$.
\end{lemma}

We will state now prove two propositions ,,approximating'' Theorem \ref{Theorem:Generalized_Isoperimetric_Inequality}.

\begin{prop}\label{p1}
Let $G = \left<S | R \right>$ be a finite presentation such that all elements of $R$ are reduced words of length $4$. Suppose that for some constant $C'$ all van Kampen diagrams $D$ with respect to this presentation satisfy

$$|\partial D| \geq C' 4 |D|.$$
Choose any $K, \varepsilon > 0$. Take $A$ large enough to satisfy $\varepsilon A > 2 (1 +  \frac{2}{C'} \log (\frac{7 A}{6 C'}) + 2 K)$. Suppose that for some $C >0$ all diagrams with $K$--small legs $Y$ having the disc basis of boundary at most $4A$ satisfy:

$$|\parti Y| \geq C 4 |Y|.$$
Then all diagrams with $K$--small legs $Y$ having the disc basis of boundary at most $\frac{14}{3} A$ satisfy:

$$|\parti Y| \geq (C - \varepsilon) 4|Y|$$
\end{prop}

\begin{proof}
Let $Y$ be the diagram with $K$--small legs such that its disc basis $Z$ has the boundary length between $4 A$ and $\frac{14}{3} A $. By Lemma \ref{cut} we can perform a partition of $Z$ into two reduced disc diagrams $Z'$ and $Z''$ such that: $|\partial Z'|, |\partial Z''| \geq \frac{1}{4} |\partial Z|$ and $|\partial Z' \cap \partial Z''| \leq 4 +  \frac{8}{C'} \log (|Z|)$. The number of $1$--cells that belong both to $Z$ and one of the legs $Z_i$ is bounded by $4K$, since there are no more than $4K$ $1$--cells in $Z_i$. We define $Y'$ to be the union of $Z'$ and all legs of $Y$ that are adjacent to $Z'$, and $Y''$ to be the union of $Z''$ and all legs of $Y$ adjacent to it. Hence we can perform a partition of $Y$ into two diagrams $Y'$ and $Y''$ with $K$--small legs such that $|\parti Y| \geq |\parti Y'| + |\parti Y''| - 2 (4 +  \frac{8}{C'} \log (|Z|) + 8K)$.

Note that $|\partial Z'|, |\partial Z''| < 4A$ so by our assumption we know

$$|\parti Y'| \geq C 4 |Y'|$$
$$|\parti Y''| \geq C 4 |Y''|.$$
Moreover by the assumption on van Kampen diagrams we obtain that $\frac{1}{C'} \log (|Z|) < \frac{1}{C'} \log (\frac{7 A}{6 C'})$. Hence

$$|\parti Y| \geq |\parti Y'| + |\parti Y''| - 2 (4 +  \frac{8}{C'} \log (|Z|) + 8 K) $$
$$\geq C 4 (|Y'| + |Y''|) - 2 (4 +  \frac{8}{C'}  \log (\frac{7 A}{6 C'}) + 8 K).$$

We have chosen $A$ large enough so that $2 (1 + 2 \frac{1}{C'} \log (\frac{7 A}{6 C'}) + 2 K) < \varepsilon A$ so we can continue estimation

$$|\parti Y| \geq C l |Y| - 4 \varepsilon A  \geq 4(C - \varepsilon)|Y|$$
since $4|Y| \geq |\partial Z| \geq 4 A$.
\end{proof}

The last approximation to Theorem \ref{Theorem:Generalized_Isoperimetric_Inequality} is the following

\begin{prop}\label{p2}
Let $G = \left<S | R \right>$ be a finite presentation such that all elements of $R$ are reduced words of length $4$. Suppose that for some constant $C'$ all van Kampen diagrams $D$ with respect to the presentation satisfy

\begin{equation}\label{Equation:All_van_Kampen}
|\partial D| \geq C' l |D|.
\end{equation}
Choose any $K, \varepsilon > 0$. Take $A$ large enough to satisfy $\varepsilon A > 2 (1 + 2 \frac{1}{C'} \log (\frac{7 A}{6 C'}) + 2 K)$. Suppose that for some $C >0$ all diagrams $Y$ with $K$--small legs and the disc basis with boundary length at most $4 A$ satisfy:

\begin{equation}\label{Equation:All_small_legs}
|\parti Y| \geq 4 C |Y|.
\end{equation}
Then all diagrams with $K$--small legs $Y$ satisfy:
$$|\parti Y| \geq 4(C - \varepsilon)|Y|$$
\end{prop}

\begin{proof}
The assumptions of this proposition and Proposition \ref{p1} are the same. Hence by the statement of Proposition \ref{p1} we can conclude that the assumptions of Proposition \ref{p1} are fulfilled with the new parameters: $A_1 = \frac{7}{6}A, \varepsilon_1 = \varepsilon (\frac{6}{7})^{\frac{1}{2}}$ and $C_1 = C - \varepsilon$ instead of $A, \varepsilon, C$ and with the same $C'$ (these new parameters indeed satisfy $\varepsilon_1 A_1 > 2 (1 +  \frac{2}{C'} \log (\frac{7 A_1}{6 C'}) + 2 K)$).
By induction every diagram with $K-small$ legs such that its disc basis has boundary of length at most $4 A \left(\frac{7}{6} \right)^k$ satisfies

$$|\parti Y| \geq \left(C - \varepsilon \sum_{i = 0}^{k-1} \left(\frac{6}{7} \right)^{\frac{i}{2}} \right)4|Y| $$
and we conclude by the inequality $\sum_{i=0}^{\infty} \left(\frac{6}{7} \right)^{\frac{i}{2}} < 14$.
\end{proof}

\begin{proof}[Proof of Theorem \ref{Theorem:Generalized_Isoperimetric_Inequality}]
First, by Theorem \ref{Theorem:Isoperimetric_inequality} we know that all Van Kampen diagrams satisfy Equation (\ref{Equation:All_van_Kampen}) for $C=(1-2d-\varepsilon')$, for arbitrary small $\varepsilon'$. By Lemma \ref{Lemma:Local_generalized_isoperimetric} we know that for any $K, \varepsilon$ all diagrams with $K$-small legs  satisfy Equation (\ref{Equation:All_small_legs}). Hence the assumptions of the Proposition \ref{p2} are satisfied, which gives the statement.
\end{proof}

\section{Colored hypergraphs}\label{Section:Colored_hypergraphs}

In this section we we will introduce colored hypergraphs, which can be effectively used in investigating random groups in the square model. We start with

\begin{prop}\label{Proposition:Special_cells} Let $\cay$ be the Cayley complex of a random group in the square model at density $d < \frac{1}{3}$. Then w.o.p.:
\begin{enumerate}
\item There is no pair of $2$-cells in  $\cay$ having three common edges.
\item If $D$ and $E$ are $2$-cells in $\cay$ such that $|\partial D \cap \partial E| = 2$, then there is no $2$-cell $F$ in $\cay$ such that $|(\partial D \cup \partial E) \cap \partial F| > 1$.
\end{enumerate}
\end{prop}

\begin{proof}
\textit{(1)} A pair of $2$-cells in $\cay$ having three common edges forms a van Kampen diagram $\mathcal{D}$ satisfying $|\partial \mathcal{D}|=2$ and $|\mathcal{D}|=2$. Therefore $|\partial \mathcal{D}| \leq \frac{4}{3}|\mathcal{D}|$, so $\mathcal{D}$ violates Theorem \ref{Theorem:Generalized_Isoperimetric_Inequality}, which implies that w.o.p. there is no such a pair of 2-cells. \\
\textit{(2)} Suppose, on the contrary, that there exist such 2-cells $D$, $E$ and $F$. Consider a diagram $\mathcal{D}$ consisting of 2-cells $D$ and $E$. Adding a new $2$-cell to $\mathcal{D}$ by gluing it along two edges does not change the generalized boundary length of $\mathcal{D}$. Hence the diagram $Y$ that is a union of $\mathcal{D}$ and the 2-cell $F$ satisfies: $|\parti Y| = 4$, $|Y| = 3$. Therefore, $Y$ violates Theorem \ref{Theorem:Generalized_Isoperimetric_Inequality}, which states that $|\tilde{\partial{Y}}|>4(1-2d)|Y|$, so w.o.p. there is no such a triple of 2-cells.
\end{proof}

\begin{definition}\label{Definition:Distinguished} Let $\cay$ be the Cayley complex of a random group in the square model at density $\leq \frac{1}{3}$. If two $2$-cells in $\cay$ have two common edges, we say they are \textit{strongly adjacent}. If a 2-cell $D$ is strongly adjacent to some other 2-cell we say that $D$ is a \textit{distinguished} 2-cell. The common edges of two strongly adjacent 2-cells we call \textit{distinguished edges}. A 2-cell which is not a distinguished 2-cell we call a \textit{regular} 2-cell.
\end{definition}

\begin{prop}\label{Proposition:Two_distinguished_in_one_orbit}
Let $D$ and $E$ be strongly adjacent 2-cells. Then there is no element of a random group $G$ sending $D$ to $E$ under the natural action of $G$ on $\cay$. In other words, $D$ and $E$ are in different orbits of the action of $G$ on $\cay$.
\end{prop}

%\begin{proof}  %OLD PROOF
%If strongly adjacent 2-cells $D$ and $E$ are in one orbit of the $G$-action on $\cay$ then $D$ and $E$ are lifts to $\cay$ of the same 2-cell in the presentation complex $\cay / G$. This means that there is a 2-cell $\mathfrak{c}$ in the presentation complex with boundary edges $e_1,\dots e_4$, such that for some distinct $1 \leq i, j, k, l \leq 4$ edge $e_i$ is identified with $e_j$ and $e_k$ is identified with $e_l$ in the presentation complex. Hence, such 2-cell is labeled by a word consisting of at most two distinct letters. The set of such words has cardinality at most $6(2n)^2$. The probability of not drawing at random such a word to the presentation can be estimated from below by

$$\frac{\left( (2n-1)^4 - 6(2n)^2 \atop (2n-1)^{4d} \right)}{\left( (2n-1)^4 \atop (2n-1)^{4d} \right)},$$

%which tends to $1$ as $n \to \infty$, as long as $d < \frac{1}{2}$.
%\end{proof}

\begin{proof}
Suppose that there exists $g \in G$ sending $D$ to $E$. By Lemma \ref{Proposition:Special_cells} we know that a 2-cell can be strongly adjacent to at most one other 2-cell, co $g$ preserves the pair $\{D, E\}$. The diagram $Y$ consisting of 2-cells $D$ and $E$ has exactly one internal vertex (that is a vertex isolated from the boundary). If the diagram $\{D, E\}$ is preserved by $g$ the internal vertex must be also preserved, which contradicts the fact that the action of $G$ on its Cayley graph is free.
\end{proof}

\begin{defi}[Procedure of painting the Cayley complex]
By Proposition \ref{Proposition:Special_cells} we know that distinguished 2-cells form a set $\mathcal{S}$ of disjoint unordered pairs of $2$-cells. By Proposition \ref{Proposition:Two_distinguished_in_one_orbit} we know that none of these pairs lies in one orbit of $G$-action on $\cay$. We will now perform a procedure of painting distinguished 2-cells in two colors: red and blue.

In the first step we consider one of the pairs $\{D, E\} \in \mathcal{S}$ and choose arbitrarily one 2-cell from this pair. Without loss of generality we can suppose that $D$ is chosen. We say that $D$ is a \textit{red} 2-cell and $E$ is a \textit{blue} cell. Moreover we say that each element in the orbit of $D$ is \textit{red} and each element in the orbit of $E$ is \textit{blue}.

In the next step we choose another pair of strongly adjacent 2-cells $\{D_1, D_2 \}$, that is not in the orbit of the pair $\{D, E \}$. Then we paint all elements in the orbit of $D_1$ in red, and all elements in the orbit of $E_1$ in blue. We repeat this step until all distinguished 2-cells are painted. The Cayley complex is cofinite, so there are finitely many orbits of $2$-cells and therefore this painting procedure must end after a finite number of steps. The Cayley complex with painted distinguished 2-cells we call \textit{painted Cayley complex}.
\end{defi}

From now we will always assume that the Cayley complex of a random group is painted, and we will denote it by $\cay$.

We now recall the definition of a hypergraph from \cite[Definition 2.1]{ow11}.

\begin{defi}[standard hypergraph]\label{Definition:Standard_hypergraph}
Let $X$ be a connected square complex. We define a graph $\Gamma$ as follows: The set of vertices of $\Gamma$ is the set of $1$-cells of $X$. There is an edge in $\Gamma$ between two vertices if there is some $2$-cell $R$ of $X$ such that these vertices correspond to opposite $1$-cells in the boundary of $R$ (if there are several such $2$-cells we put as many edges in $\Gamma$). The $2$-cell $R$ is the $2$-cell of $X$ \textit{containing} the edge.

There is a natural map $\varphi$ from $\Gamma$ to $X$, which sends each vertex of $\Gamma$ to the midpoint of the corresponding $1$-cell of $X$ and each edge of $\Gamma$ to a segment joining two opposite points in the $2$-cell $R$. Note that the images of two edges contained in the same $2$-cell $R$ always intersect, so that in general $\varphi$ is not an embedding.

A \textit{standard hypergraph} in $X$ is a connected component of $\Gamma$. The $1$-cells of $X$ through which a hypergraph passes are \textit{dual} to it. The hypergraph $\Lambda$ \textit{embeds} if $\varphi$ is an embedding from $\Lambda$ to $X$, that is, if no two distinct edges of $\Lambda$ are mapped to the same $2$-cell of $X$.

We call the subdiagram of $X$ consisting of all open faces containing edges of the hypergraph $\Lambda$ the \textit{carrier of} $\Lambda$ and denote it by $\text{Car}(\Lambda)$.

The \textit{hypergraph segment} in $X$ is a finite path in a hypergraph immersed into $X$. The \textit{carrier of a segment} $\Lambda'$ denoted $\text{Car}(\Lambda')$ is the diagram consisting of all open 2-cells containing edges of the segment $\Lambda'$.
\end{defi}

By \cite[Lemma 5.16]{odrz} we know that standard hypergraphs are embedded trees. Moreover, \cite[Lemma 2.3]{ow11} states that for each standard hypergraph $\Gamma$ the space $\cay - \Gamma$ has exactly two connected components. %This is important since our goal is to construct the action of a random group on a space with walls. 
However the system of standard hypergraphs is not sufficient to construct a proper action of a random group on a CAT(0) cube complex, which we need to prove Theorem \ref{Corollary:Cubulating}. Our strategy of constructing such an action is based on showing that the wall space metric on the Cayley complex is equivalent to the edge path metric. This means in particular that for every number $N$ if we have two enough distant points in the Cayley complex they are separated by at least $N$ walls, which are hypergraphs in our case. In Figure \ref{Figure:Problem_with_standard} we present an example of a geodesic edge-path $\gamma$ of an arbitary length such that no standard hypergraph separates its ends $x$ and $y$. Situation presented in Figure \ref{Figure:Problem_with_standard} cannot be excluded for densities $\geq \frac{1}{4}$ in the square model, therefore we need a better system of walls, based on more systems of hypergraphs.

\begin{figure}[h]
\centering
\begin{tikzpicture}[scale=0.9, rotate=-45]

%\draw[black, thick, rotate around={45:(0,0)}] (0,0) rectangle (2,2), (2,0) rectangle (4,2);
\draw[black] (0,0) rectangle (2,2);
\draw[black] (2,2) rectangle (4,4);
\draw[black] (4,4) rectangle (6,6);

\draw[black] (8,6) rectangle (10,8);

\draw (0,2) -- (2,0);
\draw (2,4) -- (4,2);
\draw (4,6) -- (6,4);
\draw (8,8) -- (10,6);

\fill (1,1) circle (1.5pt);
\fill (3,3) circle (1.5pt);
\fill (5,5) circle (1.5pt);
\fill (9,7) circle (1.5pt);

\fill (0,0) circle (1.5pt);
\fill (10,8) circle (1.5pt);

\draw[very thick]  (0,0) -- (2,0) -- (2,2) -- (4,2) -- (4,4) -- (6,4) -- (6,6) -- (6.4,5.7); 
\draw[very thick] (7.4,5.7) -- (8,6) -- (8,8) -- (10,8);

\draw[very thick, gray]  (1,0) -- (0.5,1.5) -- (2,1);
\draw[very thick, gray]  (3,2) -- (2.5,3.5) -- (4,3);
\draw[very thick, gray]  (5,4) -- (4.5,5.5) -- (6,5);

\draw[very thick, gray]  (8,7) -- (9.5,6.5) -- (9,8);

\draw (6.9,5.7) node {$\dots$};
\draw (6.9,5.7) node[above] {$\gamma$};

\draw (0,0) node[left] {$x$};
\draw (10,8) node[right] {$y$};

\end{tikzpicture}
\caption{Arbitrary long geodesic edge-path joining points $x$ and $y$ that are not separated by a hypergraph.}
\label{Figure:Problem_with_standard}
\end{figure}
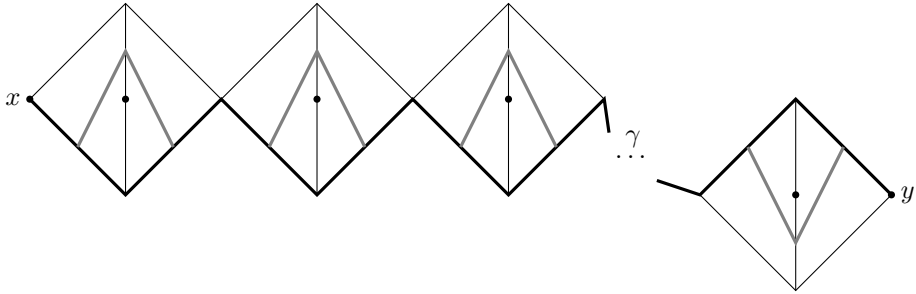

We now introduce a new type of hypergraph - red hypegraph, which is modified on red 2-cells.

\begin{defi}[red and blue hypergraph]\label{Definition:Red_hypergraph}
Let $X$ be a connected square complex that is a subcomplex of the painted Cayley complex $\cay$. We define a graph $\Gamma$ as follows: The set of vertices of $\Gamma$ is the set of $1$-cells of $X$. We put an edge in $\Gamma$ between two vertices if one the following holds:
\begin{enumerate}
\item there is some regular or blue $2$-cell $R$ of $X$ such that these vertices correspond to opposite $1$-cells in the boundary of $R$ (if there are several such $2$-cells we put as many edges in $\Gamma$
\item there is some red $2$-cell $R$ of $X$ such that one of this vertices  corresponds to distinguished edge $e$ of $R$ and the second corresponds to non distinguished edge of $R$ which is not opposite to $e$.
\end{enumerate}
The $2$-cell $R$ is the $2$-cell of $X$ \textit{containing} the edge.

There is a natural map $\varphi$ from $\Gamma$ to $X$, which sends each vertex of $\Gamma$ to the midpoint of the corresponding $1$-cell of $X$ and each edge of $\Gamma$ to a segment joining two points in the boundary $2$-cell $R$.

A \textit{red hypergraph} in $X$ is a connected component of $\Gamma$. The $1$-cells of $X$ through which a hypergraph passes are \textit{dual} to it. The red hypergraph $\Lambda$ \textit{embeds} if $\varphi$ is an embedding from $\Lambda$ to $X$.

We call the subdiagram of $X$ consisting of all open 2-cells containing edges of the red hypergraph $\Lambda$ the \textit{carrier of} $\Lambda$ and denote it by $\text{Car}(\Lambda)$.

The \textit{hypergraph segment} in $X$ is a finite path in a hypergraph immersed into $X$. The carrier of a segment $\Lambda'$ denoted $\text{Car}(\Lambda')$ is the diagram consisting of all open 2-cells containing edges of the segment $\Lambda'$.

If in the above definition we replace every occurrence of the word ``red'' by the word ``blue'' we obtain the definition of the \textit{blue hypergraph}.
\end{defi}

The comparison of the hypergraphs introduced above is presented in Figure \ref{Figure:Comparison_of_hypergraphs}.

\begin{figure}[h]
\centering
\begin{tikzpicture}[scale=0.9]

%2 komorki:
\filldraw[draw=black,fill=gray!20] (-2,0) rectangle (0,2);
\filldraw[draw=black,fill=gray!20] (0,0) rectangle (2,2);
\filldraw[draw=black,fill=gray!20] (2,0) rectangle (4,2);
\filldraw[draw=black,fill=gray!20] (0,2) rectangle (2,4);

%hipergraf

\path[black, draw] (0,2) -- (2,0);
\path[red, draw,line width=2pt] (-2,1.1) -- (0,1.1)--(0.5,1.5)--(2,1.1)--(4,1.1);
\path[black, draw,line width=2pt] (-2,0.9) -- (0,0.9)--(1.5,0.5)--(0.9,2)--(0.9,4);

%wezly
%\draw (0.5,1.5) node[above left] {$\gamma'$};
\fill (1,1) circle (1.5pt);

\draw (3,1.1) node[above] { ${\color{red} \Lambda_{\text{red}}}$};
\draw (1,3) node[right] { ${\color{black} \Lambda_{\text{st}}}$};

\draw (0,0) node[above right] {\footnotesize{red 2-cell}};
\draw (1.55,2.1) node[below] {\footnotesize{blue}};
\draw (1.6,1.8) node[below] {\footnotesize{2-cell}};

\draw (1,-1) node[below] {a)};

%druga czesc

\begin{scope}[shift={(8,0)}]

%2 komorki:
\filldraw[draw=black,fill=gray!20] (-2,0) rectangle (0,2);
\filldraw[draw=black,fill=gray!20] (0,0) rectangle (2,2);
\filldraw[draw=black,fill=gray!20] (2,0) rectangle (4,2);
\filldraw[draw=black,fill=gray!20] (0,2) rectangle (2,4);

%hipergraf

\path[black, draw] (0,2) -- (2,0);
\path[blue, draw,line width=2pt] (-2,1.1) -- (0,1.1)--(1.3,0.7)--(2,1.1)--(4,1.1);
\path[black, draw,line width=2pt] (-2,0.9) -- (0,0.9)--(1.5,0.5)--(0.9,2)--(0.9,4);

%wezly
%\draw (0.5,1.5) node[above left] {$\gamma'$};
\fill (1,1) circle (1.5pt);

\draw (3,1.1) node[above] { ${\color{blue} \Lambda_{\text{blue}}}$};
\draw (1,3) node[right] { ${\color{black} \Lambda_{\text{st}}}$};

\draw (0,0) node[above right] {\footnotesize{red 2-cell}};
\draw (1.55,2) node[below] {\footnotesize{blue}};
\draw (1.6,1.7) node[below] {\footnotesize{2-cell}};

\draw (1,-1) node[below] {b)};

\end{scope} 

\end{tikzpicture}
\caption{a) The comparison of a red hypergraph $\Lambda_{\text{red}}$ and a standard hypergraph $\Lambda_{\text{st}}$, b) The comparison of a blue hypergraph $\Lambda_{\text{blue}}$ and a standard hypergraph $\Lambda_{\text{st}}$.}
\label{Figure:Comparison_of_hypergraphs}
\end{figure}
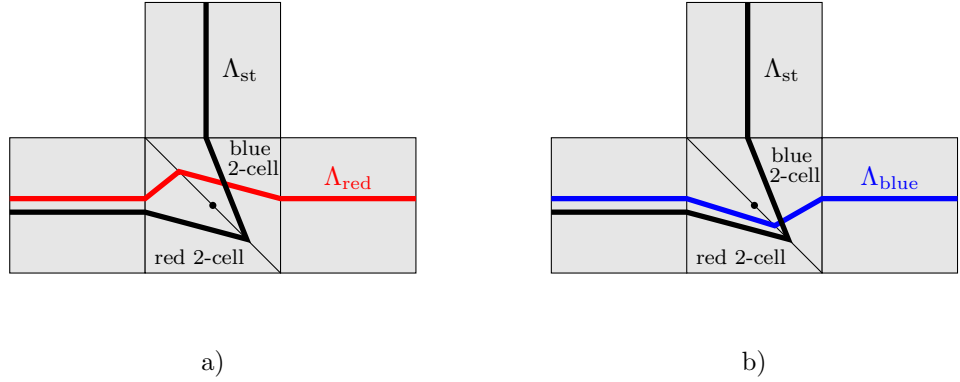

\subsection{Hypergraphs are embedded trees}

Now we are going to prove

\begin{theo}\label{Theorem:Red_hypergraphs_are_trees}
In the square model at density $d < \frac{1}{3}$ w.o.p. all red and blue hypergraphs in $\cay$ are embedded trees.
\end{theo}

To investigate standard and red hypergraphs we need a notion of a collared diagram. Our definition is a slight modification of the one introduced by Ollivier and Wise in \cite[Definition 3.2]{ow11}.

\begin{definition}\label{def:coll}(collared diagram)
Let $D$ be a Van Kampen with the following properties:
\begin{enumerate}
\item there is an external $2$-cell $C$ called a corner of $D$
\item there is a segment $\lambda \rightarrow D \rightarrow \cay$ of a hypergraph (standard, red or blue) of length at least $2$
\item the first and the last edge of $\lambda$ lie in $C$, and no other edge lies in $C$
\item $\lambda$ passes through every other external $2$-cell of $D$ exactly once
\item $\lambda$ does not pass through any internal $2$-cell of $D$
\end{enumerate}
We call $D$ a \textit{collared diagram} and $\lambda$ a \textit{collaring segment}. We also say that $D$ is \textit{cornerless} if the first and the last edge of $\lambda$ coincide in $C$ (in which case the hypergraph cycles).
\end{definition}

Now we will use the additional color-structure of a Painted Cayley complex to perform the following procedure of adding horns to a diagram.

\begin{definition}[adding horns procedure]\label{Definition:Adding_horns}
Let $\mathcal{D}$ be a van Kampen diagram in $\cay$. Let $\partial_2(\mathcal{D})$ be the set of all 2-cells in $\mathcal{D}$ that contain a boundary edge of $\mathcal{D}$. Let $\partial_2^s(\mathcal{D})$ be the set of all distinguished 2-cells $\mathfrak{c}$ in $2-\partial(\mathcal{D})$ such that a 2-cell $\mathfrak{c}'$, that is strongly adjacent to $\mathfrak{c}$, does not belong to $\mathcal{D}$.

We define a \textit{diagram with horns} $\mathcal{D}$ in the following way: for each $\mathfrak{c} \in \partial_2(\mathcal{D})$ we glue to a diagram $\mathcal{D}$ a 2-cell $\mathfrak{c}'$ that is strongly adjacent to $\mathfrak{c}$ (gluing is in a way that these 2-cells are glued in $\cay$).
\end{definition}

If that will be clear we will call diagrams with horns just diagrams, for shortness.

\begin{lem}\label{Lemma:Pairity_of_the_boundary}
Let $\mathcal{D}$ be a diagram with horns. Then $|\tilde{\partial}(\mathcal{D})|$ (the generalized boundary length) is an even number.
\end{lem}

\begin{proof}
The statement results from the fact that every 2-cell of $\mathcal{D}$ has even number of edges and every gluing of 2-cells along an edge decreases the generalized boundary length by 2.
\end{proof}

\begin{rem}\label{Remark:Shape_of_diagrams_with_horns}
A diagram with horns is a union of a disc diagram $D$ and a set $Z$ of $2$-cells, each of them being adjacent to the boundary of $D$ along at least one vertex and such that no two elements of $Z$ share an edge.
\end{rem}

\begin{proof}
The only thing that we need to prove is that no two elements of $Z$ share an edge. This results by Proposition \ref{Proposition:Special_cells}.
\end{proof}

\begin{cor}[of Remark {\ref{Remark:Shape_of_diagrams_with_horns}}]
Collared diagrams with horns w.o.p. satisfy the assumptions of Theorem \ref{Theorem:Generalized_Isoperimetric_Inequality}.
\end{cor}

The following theorem shows the relation between collared diagrams and hypergraphs

\begin{theo}[reformulation of {\cite[Theorem 3.5]{ow11}}]\label{Theorem:Existence_of_collared_diagram}
Let $\Lambda$ be a hypergraph in $\cay$. The following conditions are equivalent:
\begin{enumerate}
\item $\Lambda$ is an embedded tree.
\item There is no collared diagram collared by a segment of $\Lambda$.
\end{enumerate}
\end{theo}

There is a definition of a diagram collared by a hypergraph in \cite[Definition 3.11]{ow11}. However, in the  paper \cite{ow11} this definition was stated for a general $l$-gon complex and for the standard hypergraphs that join antipodal midpoints of faces they are passing through. Definition \ref{Definition:Collared_by_hypergraph_path_diagram} is a modification of \cite[Definition 3.11]{ow11}: we use it for $l=4$ and for hypergraphs defined in a different way, but the fact the hypergraph joins the antipodal midpoints of edges in $\cay$ is not necessary to provide the proof of Theorem \ref{Theorem:Existence_of_collared_diagram}, in fact the proof can be repeated, as well, in case of colored hypergraphs.

The next theorem is motivated by Theorem \ref{Theorem:Existence_of_collared_diagram} and shows the relation between collared diagrams with horns and colored hypergraphs.

\begin{theo}\label{Theorem:Collared_diagram_with_horns}
Let $\Lambda$ be a red or blue hypergraph in $\cay$. If \item $\Lambda$ is not an embedded tree, then there is a collared diagram with horns collared by a segment of $\Lambda$.
\end{theo}

\begin{proof}
By Theorem \ref{Theorem:Existence_of_collared_diagram} we know that there exists collared diagram collared by a segment of $\Lambda$. We obtain the statement of Theorem \ref{Theorem:Collared_diagram_with_horns} by applying the procedure of adding horns to a diagram described in Definiton \ref{Definition:Adding_horns}
\end{proof}

\begin{proof}[Proof of Theorem \ref{Theorem:Red_hypergraphs_are_trees}]

Suppose, on the contrary, that there is a red or blue hypergraph that is not an embedded tree. By Theorem \ref{Theorem:Collared_diagram_with_horns} we know that there is a collared diagram with horns $E$ collared by some segment $\lambda$. Let $n= |\parti E|$.  For $\varepsilon < 2\left(\frac{1}{3} - d \right)$ by Theorem \ref{Theorem:Generalized_Isoperimetric_Inequality} we have w.o.p.

\begin{equation}\label{nier1}
n=|\parti E| \geq 4|E|(1-2d-\varepsilon) > \frac{4}{3}|E|
\end{equation}

There are two possibilities: either $E$ is cornerless or not. First consider the case where $E$ is cornerless. Then every external $2$-cell, which is not red or blue has exactly one external edge. Each pair of strongly adjacent 2-cells gives a contribution to $|\parti E|$ equal $2$, independently of the way the blue or red 2-cell is glued to the diagram  $E$. Hence $|E| \geq n$. By (\ref{nier1}) we know that with overwhelming probability all collared cornerless diagrams with horns satisfy:

\begin{equation}\label{nier2}
n > \frac{4}{3}n,
\end{equation}
which is a contradiction. Therefore, with overwhelming probability there are no such diagrams.

Let us now consider the case where the diagram $E$ is not cornerless. The only difference between the previous case is that a corner can have two external edges. Therefore $|E| \geq n-1$. We have two possibilities $|E| \geq n$ or $|E|=n-1$. If $|E| \geq n$ we again obtain (\ref{nier2}), which is a contradiction. Hence, the only remaining case is where $|\parti E| = n$ and $|E| = n-1$. Again we use (\ref{nier1}) to obtain:

$$n > \frac{4}{3}(n-1).$$

It can be easily seen that for $n > 3$ this is not possible. So we only have to exclude the diagram $|E| = 2, |\parti E|=3$. But there are no diagrams with odd generalized boundary length, according to Lemma \ref{Lemma:Pairity_of_the_boundary}.
\end{proof}

\begin{lem}\label{Lemma:Connected_components}
Suppose that a red hypergraph or a blue hypergraph $\Lambda$ is an embedded tree in $\cay$. Then $\cay - \Lambda$ consists of two connected components.
\end{lem}

\begin{proof}
The proof is analogous to the proof of \cite[Lemma 2.3]{ow11}.
\end{proof}

\section{Metrically proper action on a space with walls}

In this section our goal is to prove the following

\begin{theo}\label{Theorem:Proper_action}
Let $G$ be a random group in the square model at density $d < \frac{3}{10}$. Let $\cay$ be the Cayley complex of $G$. For $x, y \in \cay$ we denote by $d_{\emph{wall}}(x,y)$ the number of hypergraphs $\Lambda$ (standard, red and blue) such that $x$ and $y$ lie in the different connected components of $\cay - \Lambda$. Then w.o.p.

$$ d_{\emph{wall}}(x,y) \geq \lfloor \frac{1}{15} d_{\cay^{(1)}}(x,y) \rfloor,$$

where $d_{\cay^{(1)}}$ denotes the edge metric on the 1-skeleton of $\cay$.
\end{theo}

The meaning of the foregoing theorem comes from the following

\begin{proof}[Proof of Theorem \ref{Theorem:Cubulating}]
By Lemma \ref{Lemma:Connected_components} hypergraphs are embeddes trees, and by Lemma \ref{Lemma:Connected_components} and \cite[Lemma 2.3]{ow11} they separate the Cayley complex into two connected components, so they provide a structure of a space with walls on it. Every discrete group acts properly on its Cayley complex, and by Theorem \ref{Theorem:Proper_action} the wall metric and edge metric on the Cayley complex are equivalent. Therefore a group acts properly on a space with walls. Then by \cite{walls} there is an action of a random group on a CAT(0) cube complex.
\end{proof}

Before we provide the proof of Theorem \ref{Theorem:Proper_action} we need several facts about the geometry of hypergraphs. For simplicity, until the end of this section we will denote by $\cay$ the painted Cayley complex of a random group in the square model at density $d<\frac{3}{10}$ and by $\cay^{(1)}$ the 1-skeleton of $\cay$ (i.e. the Cayley graph of a random group).

\begin{lem}\label{Lemma:Adjacent_to_special_cell}
Let $D$ be a distinguished 2-cell in $\cay$, let $E$ be a 2-cell having exactly one common edge with $D$ and let $F$ be a 2-cell different than $D$ and having at least one common edge with $E$. Then w.o.p. $E$ and $F$ are regular 2-cells.
\end{lem}

\begin{proof}
Suppose, on the contrary, that $E$ is a distinguished 2-cell. Then let $D'$ and $E'$ be the 2-cells strongly adjacent to $D$ and $E$ respectively. By the assumption we know that $D' \neq E$ and $E' \neq D$ (otherwise $D$ and $E$ would be strongly adjacent). By Proposition \ref{Proposition:Special_cells} we know that $D' \neq E'$. Therefore there are two pairs of distinguished 2-cells $\{D, D'\}$ and $\{E, E'\}$ which have at least one common edge. Consider the diagram $\mathcal{D}$ consisting of 2-cells $D, D', E, E'$, glued in the way they are glued in $\cay$. We will show that $|\tilde{\partial} D| \leq 6$. Note, that every diagram $\mathcal{A}$ consisting of two strongly adjacent 2-cells satisfies: $|\mathcal{A}| = 4$ and $|\tilde{\partial} \mathcal{A}| = 4$. Gluing two such diagrams, consisting of disjoint sets of 2-cells, along one edge decreases the generalized boundary length by 2, so $|\tilde{\partial} \mathcal{D}| \leq 6$. Therefore, the diagram $\mathcal{D}$ violates the inequality $|\tilde{\partial} \mathcal{D}| \geq |\mathcal{D}|4(1 - 2d)$, so by Theorem \ref{Theorem:Generalized_Isoperimetric_Inequality} w.o.p. there is no such a diagram $\mathcal{D}$. Therefore $E$ is a regular 2-cell.

Now suppose, on the contrary, that $F$ is a distinguished 2-cell. Let $F'$ be the 2-cell strongly adjacent to $F$. Again by Propistion \ref{Proposition:Special_cells}, shows that $F' \neq D$, $D' \neq F$ and $F' \neq D'$. Moreover, from the previous part of the proof we know that with w.o.p. $E$ is a regular 2-cell, so $F' \neq E$. Hence, there is a diagram $\mathcal{E}$, which is a union of a pair $\{D, D'\}$ of strongly adjacent 2-cells, 2-cell $E$, and a second pair of strongly adjacent 2-cells $\{F, F'\}$. Analogously as above, we can prove that $\mathcal{E}$ satisfies: $|\mathcal{E}| = 5$ and $|\tilde{\partial} \mathcal{E}| \leq 8$. Therefore the diagram $\mathcal{E}$ violates the inequality $|\tilde{\partial} \mathcal{E}| \geq |\mathcal{E}|4(1 - 2d)$, so by Theorem \ref{Theorem:Generalized_Isoperimetric_Inequality} w.o.p. there is no such a diagram $\mathcal{E}$. Hence $F$ is a regular 2-cell.
\end{proof}

Now, we introduce the notion of a diagram collared by hypergraph and path.
\begin{defi}\label{Definition:Collared_by_hypergraph_path_diagram}
Let $\lambda$ be a segment of a hypergraph (red, blue or standard) and let $\gamma$ be an edge-path in $\cay^{(1)}$. Let $D$ be a van Kampen diagram with the following additional properties:
\begin{enumerate}
\item $\lambda$ passes only through external 2-cells of $D$
\item every external edge of $D$ belongs to the boundary of the carrier of $\lambda$ or belongs to $\gamma$.
\end{enumerate}
We call $D$ a \textit{diagram collared by hypergraph and path}. After applying the procedure of adding horns(see Definition \ref{Definition:Adding_horns}) to $D$ we obtain a \textit{diagram with horns collared by hypergraph and path}. We call $\lambda$ a \textit{collaring hypergraph segment} and $\gamma$ a \textit{collaring edge-path}.
\end{defi}

There is a definition of a diagram collared by hypergraphs and paths in \cite[Definition 3.11]{ow11}. In the paper \cite{ow11} this definition was stated for a general polygonal complex and for the hypergraphs that join antipodal midpoints of faces they are passing through. Definition \ref{Definition:Collared_by_hypergraph_path_diagram} is a modification of \cite[Definition 3.11]{ow11}: we use it for $l=4$, one path and one hypergraph segment, and for hypergraphs defined in a different way, but the fact the hypergraph joins the antipodal midpoints is not necessary to provide the proofs of upcoming lemmas. The proofs of lemmas from \cite{ow11} that we will use can be as well repeated for red and blue hypergraphs. The following lemma is then a slight modification of \cite[Lemma 3.17]{ow11}.

\begin{lem}\label{Lemma:Diagram_collared_by_path_and_hypergraph}
Let $\Lambda$ be a hypergraph in $\cay$ and $\gamma$ be an edge path in $\cay^{(1)}$. Suppose, that $\gamma$ intersects $\Lambda$ in two consecutive points $x, y$. Then there is a diagram with horns collared by $\Lambda$ and a segment of $\gamma$ bounded by points $x$ and $y$ (here $\gamma$ is an edge path which starts and ends at``midedge vertices'' corresponding to vertices of $\Lambda$.
\end{lem}

\begin{proof}
First we prove that there exists a diagram collared collared by $\Lambda$ and a segment of $\gamma$ bounded by points $x$ and $y$. The proof of this fact is analogous to the proof of \cite[Lemma 3.17]{ow11}. Then we apply a procedure of adding horns (Definition \ref{Definition:Adding_horns}) to that diagram.
\end{proof}

\begin{lem}\label{Lemma:Geodesic_breaks_away_from_carrier}
Let $\gamma$ be a geodesic edge-path in $\cay^{(1)}$ and let $\Lambda$ be a hypergraph in $\cay$. Suppose that there is a finite, nonempty geodesic segment $\gamma' \subset \gamma$, such that the first and the last vertex of $\gamma'$ lie in $\emph{Car}(\Lambda)$ but no other vertices of $\gamma'$ lie in $\emph{Car}(\Lambda)$ and no edge of $\gamma'$ lie in $\emph{Car}(\Lambda)$. Then w.o.p. $\gamma'$ has length 2, and there is the following diagram $\mathcal{D}$ consisting of three 2-cells immersed in $\cay$:

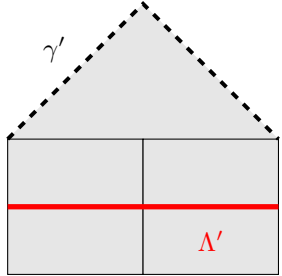
\begin{figure}[h]\label{fig2}
\centering
\begin{tikzpicture}[scale=1.8]

%2 komorki:
\fill[fill=gray!20] (0,1)--(1,2)--(2,1)--cycle;
\filldraw[draw=black,fill=gray!20] (0,0) rectangle (1,1);
\filldraw[draw=black,fill=gray!20] (1,0) rectangle (2,1);

%hipergraf

\path[red, draw,line width=2pt] (0,0.5) -- (2,0.5);

\draw (1.5,0.25) node { ${\color{red} \Lambda'}$};

%geodezyjna
\draw[draw,line width=1.5pt, dashed] (0,1) -- (1,2) -- (2,1);

%wezly
\draw (0.5,1.5) node[above left] {$\gamma'$};

\end{tikzpicture}
\caption{The diagram $\mathcal{D}$.}
\end{figure}

Moreover, every 2-cell of $\mathcal{D}$ is a regular 2-cell.
\end{lem}

\begin{defi}
A diagram $D$ having the same shape as the one presented in Figure \ref{fig2} will be called a \textit{house diagram}. The geodesic segment $\gamma'$ we be called a \textit{roof (of a house diagram)}.
\end{defi}

\begin{proof}
Let $e_0, e_1$ be some edges dual to $\Lambda$ with endpoints in the first and the last vertex of $\gamma'$. By Lemma \ref{Lemma:Diagram_collared_by_path_and_hypergraph} we know that there is a diagram $\mathcal{D}$ with horns collared by the edge-path $\{e_0, \gamma', e_1\}$ and a segment $\Lambda'$ of $\Lambda$. Suppose that the length of this hypergraph segment equals $k$ and the length of $\gamma'$ equals $l$. Let $k'$ denote the number of horns in $\mathcal{D}$. Note that
\begin{equation}
|\tilde{\partial} \mathcal{D}| \leq k+k'+l+2.
\end{equation}
Note that every edge of $\gamma'$ lies in the boundary of some 2-cell in $\mathcal{D}$ that is not in the carrier of $\Lambda$, since otherwise there would be an inner vertex of $\gamma'$ lying in $\emph{Car}(\Lambda)$. Moreover, more than 2 edges of $\gamma'$ cannot lie in the boundary of one 2-cell in $\mathcal{D}$, since otherwise $\gamma'$ would not be a geodesic. Therefore, there are at least $\left \lceil \frac{l}{2} \right \rceil$ 2-cells in $\mathcal{D}$ not lying in $\emph{Car}(\Lambda)$. Hence
\begin{equation}
|\mathcal{D}| \geq k+k'+ \left \lceil \frac{l}{2} \right \rceil.
\end{equation}

Moreover, $l \geq 2$, because there is at least one vertex of $\gamma'$ not lying in the carrier of $\Lambda$ and $l \leq k$ because $\gamma$ is a  geodesic. By Lemma \ref{Lemma:Pairity_of_the_boundary} the number $k+k'+l$ must be even. Combining this with Theorem \ref{Theorem:Generalized_Isoperimetric_Inequality} gives the following system of conditions, which is w.o.p. satisfied by $k$, $k'$ and $l$:

\begin{equation}\label{Equation:Set_of_conditions}
\begin{cases} k+k'+l+2 > \frac{8}{5}\left( k+k'+ \left \lceil \frac{l}{2} \right \rceil \right)
\\ l \leq k, l \geq 2, k \geq 1 \iff 2 \leq l \leq k
\\ k+k'+l \text{  is an even number} \end{cases}
\end{equation}

Now we will find the possible values of $l$ (see \ref{Equation:Set_of_conditions}).

If $l$ is even then we can rewrite the first inequality in (\ref{Equation:Set_of_conditions}) in form $k+k'+l + 2 > \frac{8}{5}(k+k' + \frac{l}{2})$. This is equivalent to $3(k+k') < 10+l$. Combining this with the second inequality in (\ref{Equation:Set_of_conditions}), that is  $l \leq k$, we obtain $k < 5$. Therefore $l \leq 4$, since $l$ is even. Hence for even $l$ possible solutions are: $l=4$ and $l=2$.

If $l$ is odd then the first inequality in (\ref{Equation:Set_of_conditions}) becomes $k+k'+l + 2 > \frac{8}{5}(k+k' + \frac{l+1}{2})$, which is equivalent to $3(k+k') < l + 6$. Combining this with $l \leq k$ we get $k < 3$. Therefore $l \leq 1$ since $l$ is odd. This contradicts $l \geq 2$. Hence, there are no solutions with odd $l$.

First, consider the case, where $l=2$. Then
$$k + k' + 4 > \frac{8}{5}(k+k'+1),$$

so $k+k' < 4$. Combining this with the parity of $k+k'+l$ we obtain that only possible solution for $l=2$ is $k+k'=2$. Hence, there are two subcases to consider: $k'=0$ and $k'=1$. If $k'=1$ then the segment $\Lambda'$ consists of a distinguished 2-cell $\mathfrak{c}$, which is strongly adjacent to some 2-cell $\mathfrak{c}'$:

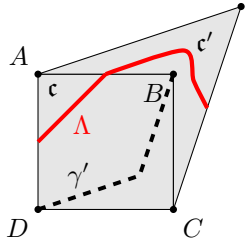
\begin{figure}[h]
\centering
\begin{tikzpicture}[scale=0.9]

%2 komorki:
\filldraw[draw=black,fill=gray!20] (-1,-1) rectangle (1,1);
%\fill[fill=gray!20] (-1,1) -- (2,2) -- (1,-1) -- (1,1) --cycle;
\draw [draw=black, fill=gray!20]
       (-1,1) -- (2,2) -- (1,-1) -- (1,1) --cycle;

%hipergraf
\path[red, draw,line width=1.5pt] (-1,0) -- (0,1);
\draw[red, preaction={draw, red, thick, double distance=0pt}] plot[smooth] coordinates { (0,1) (1.15,1.35) (1.3,0.9) (1.5,0.5)};

%geodezyjna
\draw[draw,line width=1.5pt, dashed] (-1,-1) -- (0.5,-0.5) -- (1,1);

%wezly
\fill (-1,1) circle (1.5pt) node[above left] {$A$};
\fill (1,1) circle (1.5pt) node[below left] {$B$};
\fill(1,-1) circle (1.5pt) node[below right] {$C$};
\fill (-1,-1) circle (1.5pt) node[below left] {$D$};
\fill (2,2) circle (1.5pt);

\draw (-1, 1) node[below right] {$\mathfrak{c}$};
\draw (1.2, 1.2) node[above right] {$\mathfrak{c}'$};

\draw (-0.4, -0.50) node {$\gamma'$};
\draw (-0.35,0.25) node { ${\color{red} \Lambda}$};

\end{tikzpicture}
\caption{The diagram $\mathcal{D}$ with one of possible hypergraphs $\Lambda$ and geodesic segments $\gamma'$.}
\label{Figure:Breaks_away_from_carrier_Fig1}
\end{figure}

The geodesic segment $\gamma'$ joins vertices $A$ and $C$ or $B$ and $D$, since otherwise $\gamma$ would not be geodesic. It can be easily seen that the 2-cell $\mathfrak{c}'$ also belongs to the carrier of $\Lambda$. Therefore $\gamma'$ cannot run along the boundary path of $\mathfrak{c}'$, since it has at least one vertex that lies outside the carrier of $\Lambda$. Hence, there must be at least one more 2-cell in $\mathcal{D}$. This gives us $|\tilde{\partial} \mathcal{D}| = 4$ and $|\mathcal{D}| \geq 3$, which contradicts Theorem \ref{Theorem:Generalized_Isoperimetric_Inequality}.

The second subcase (of the case $l=2$) is $k'=0$ and $k=2$. In that subcase $\emph{Car}(\Lambda')$ consists of two regular 2-cells (see Figure \ref{Figure:Breaks_away_from_carrier_Fig2} a) ).

\begin{figure}[h!]
\centering
\begin{tikzpicture}[scale=1.5]

% czesc a)

%2 komorki:
\filldraw[draw=black,fill=gray!20] (0,0) rectangle (1,1);
\filldraw[draw=black,fill=gray!20] (1,0) rectangle (2,1);

%hipergraf

\path[red, draw,line width=2pt] (0,0.5) -- (2,0.5);

\draw (1.5,0.25) node { ${\color{red} \Lambda'}$};

%wezly
\fill (0,1) circle (1.5pt) node[above left] {$A$};
\fill (1,1) circle (1.5pt) node[above] {$B$};
\fill (2,1) circle (1.5pt) node[above right] {$C$};
\fill (2,0) circle (1.5pt) node[below right] {$D$};
\fill (1,0) circle (1.5pt) node[below] {$E$};
\fill (0,0) circle (1.5pt) node[below left] {$F$};

\draw (1,-1) node {a)};

% czesc b)

\begin{scope}[shift={(5,0)}]

%2 komorki:
\fill[fill=gray!20] (0,1)--(1,2)--(2,1)--cycle;
\filldraw[draw=black,fill=gray!20] (0,0) rectangle (1,1);
\filldraw[draw=black,fill=gray!20] (1,0) rectangle (2,1);

%hipergraf

\path[red, draw,line width=2pt] (0,0.5) -- (2,0.5);

\draw (1.5,0.25) node { ${\color{red} \Lambda'}$};

%geodezyjna
\draw[draw,line width=1.5pt, dashed] (0,1) -- (1,2) -- (2,1);

%wezly
\draw (0.5,1.5) node[above left] {$\gamma'$};
\fill (0,1) circle (1.5pt) node[above left] {$A$};
\fill (1,1) circle (1.5pt) node[above] {$B$};
\fill (2,1) circle (1.5pt) node[above right] {$C$};
\fill (2,0) circle (1.5pt) node[below right] {$D$};
\fill (1,0) circle (1.5pt) node[below] {$E$};
\fill (0,0) circle (1.5pt) node[below left] {$F$};

\draw (1,-1) node {b)};

  \end{scope}

\end{tikzpicture}
\caption{The diagram $\mathcal{D}$.}
\label{Figure:Breaks_away_from_carrier_Fig2}
\end{figure}
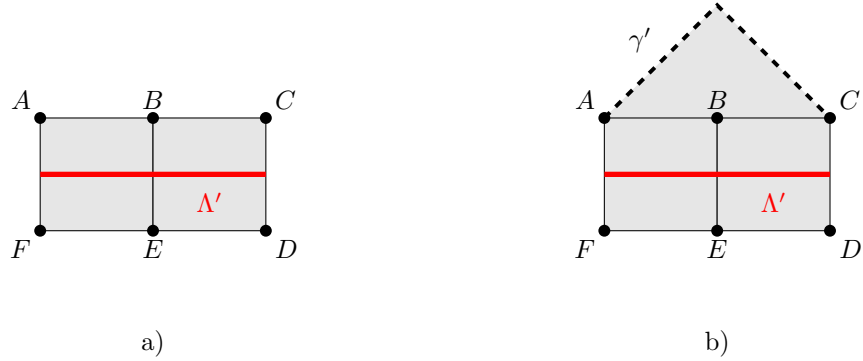

Note, that $\gamma'$  can join only vertices $A$ with $C$ or $D$ with $F$, since in all other cases $\gamma'$ is either not a geodesic or intersects $\Lambda$, according to Lemma \ref{Lemma:Connected_components}. After relabelling vertices, we can always obtain a diagram $\mathcal{D}$ presented in Figure \ref{Figure:Breaks_away_from_carrier_Fig2} b), as desired. Moreover, every 2-cell in $\mathcal{D}$ is a regular 2-cell, since $\mathcal{D}$ does not have any horns. There are exactly three 2-cells in $\mathcal{D}$ since by \ref{Theorem:Generalized_Isoperimetric_Inequality} there are no diagrams with the boundary equal 6 consisting of more than three 2-cells.

Now, consider the case where $l=4$. Then
$$k+k'+6 \geq \frac{8}{5}(k+k'+2),$$

so $k+k'\leq 4$. We start with the subcase, where $k+k'=4$. We know that $|\tilde{\partial} \mathcal{D}| \leq k+k'+l = 10$ and $|\mathcal{D}| \geq k+k'+\frac{l}{2} = 6$. By Theorem \ref{Theorem:Generalized_Isoperimetric_Inequality} and the fact that $|\tilde{\partial} D| \leq 10$ we obtain that $|\mathcal{D}|=6$ and $|\tilde{\partial} \mathcal{D}| = 10$. Let $\mathcal{C} = \mathcal{D} \cap \text{Car}(\Lambda)$. The diagram $\mathcal{D}-\mathcal{C}$ consists of two 2-cells, call them $A$ and $B$. Note that all edges of $\gamma'$ belong to $A \cup B$. Therefore 2-cells $A$ and $B$ have no common edges, since  otherwise there will  be no two points in $A \cup B$ within distance 4. This means, that there are three vertices of $\gamma'$ belonging to $\mathcal{C}$. This contradicts our assumption on $\gamma'$ so there is no diagram with $k+k'=4$ satisfying the assumptions of the Lemma \ref{Lemma:Geodesic_breaks_away_from_carrier}.

Now consider the case, where $k+k'=2$ and $l=4$.  If $k+k'=2$ then the carrier of $\Lambda'$ consists of maximally two 2-cells, so every pair of vertices in this carrier lies in the distance $d_{\cay^{(1)}}$ equal at most 3. Therefore $\gamma'$ cannot join any vertices of the carrier of $\Lambda'$. Hence, there is no diagram $\mathcal{D}$ with $k+k'=2$ and $l=4$.

There are no other possibilities, because of the parity of $k+k'+l$.
\end{proof}

\begin{defi}
Let $\lambda$ be a segment of a hypergraph. A diagram consisting of two strongly adjacent 2-cells, such that $\lambda$ passes through this pair will be called  a \textit{divided tile (of $\lambda$)} we. A 2-cell in the carrier of $\lambda$ that is not a regular $2$-cell but is not an element of a divided tile of $\lambda$ will be called an \textit{isolated 2-cell}. A \textit{tile} is a divided tile or a regular 2-cell.
\end{defi}

%The definition of square is motivated by the fact that the boundary of it quadrangle and both red and blue hypergraph join antipodal midpoints of it.

\begin{lem}[hypergraphs do not osculate]\label{Lemma:Osculating}
Let $\Lambda'$ be a hypergraph segment in $\cay^{(1)}$. Suppose that one of the following conditions is satisfied
\begin{enumerate}[a)]
\item $\Lambda'$ is a segment of a red hypergraph and there are no isolated red 2-cells of $\Lambda'$,
\item $\Lambda'$ is a segment of a blue hypergraph and there are no isolated blue 2-cells of $\Lambda'$,
\item $\Lambda'$ is a segment of a standard hypergraph and there is no divided tile of $\Lambda'$.
\end{enumerate}Then w.o.p. holds
\begin{enumerate}
\item For every vertex $v$ of $\text{Car}(\Lambda')$ there is only one edge in $\text{Car}(\Lambda')$ ending in $v$, dual to $\Lambda'$ that is not an internal edge of a divided tile (the situation presented in Figure \ref{Figure:Osculating} a) does not occur). In other words: if there is a vertex in $\text{Car}(\Lambda')$ being the endpoint of two edges dual to $\Lambda'$ then one of these edges must be an internal edge of a divided tile. 
\item There is no edge $e$ in $\cay^{(1)}$ such that its vertices lie in $\text{Car}(\Lambda')$ but $e$ does not lie in $\text{Car}(\Lambda')$, see Figure \ref{Figure:Osculating} b).
\end{enumerate}

\begin{figure}[h!]
\centering
\begin{tikzpicture}[scale=1.2]

\draw  plot[smooth, tension=.7] coordinates {(0,-3) (0.0071,-2.3246) (0.2554,-1.8975) (0.244,-1.3927) (-0.0354,-1.0741) (-0.4544,-0.7399) (-0.7887,-0.3747) (-0.9834,0.1302) (-0.8832,0.7283) (-0.4981,1.0161) (0.1226,0.978) (0.6037,0.7162) (0.852,0.3282) (0.8675,-0.2615) (0.7433,-0.7581) (0.5261,-1.1616) (0.2778,-1.3788) (0.4795,-1.8133) (0.883,-2.2323) (1.1623,-2.8841) };

\draw  plot[smooth, tension=.7] coordinates {(-0.2859,-3.0525) (-0.3065,-2.2978) (-0.0171,-1.8429) (0.0036,-1.4191) (-0.1825,-1.2537) (-0.6167,-0.9849) (-1.0612,-0.5507) (-1.2654,0.111) (-1.1285,0.8966) (-0.5626,1.2481) (0.2024,1.1835) (0.833,0.8154) (1.1225,0.2675) (1.1431,-0.3528) (0.9571,-0.9937) (0.6779,-1.3659) (0.5022,-1.4486) (0.6469,-1.6967) (1.1431,-2.1309) (1.3706,-2.8339)};

%\draw (0,-3) node (v1) {};
%\draw (v1) -- cycle;
%\draw (0,-3) -- (-0.0065,-3.0058) node (v2) {};

%\draw  -- (v2);
\draw (-0.3015,-2.2956) -- (0.0118,-2.3211);
\draw (-0.0137,-1.8367) -- (0.2667,-1.8913);
\draw (0.0227,-1.4797) -- (0.2922,-1.4542);
\draw (-0.174,-1.2612) -- (-0.0356,-1.0681);
\draw (-0.6183,-0.988) -- (-0.4508,-0.7367);
\draw (-1.0554,-0.5473) -- (-0.7859,-0.3834);
\draw (-1.2521,0.1119) -- (-0.9789,0.1229);
\draw (-1.1246,0.895) -- (-0.8733,0.7311);
\draw (-0.5564,1.241) -- (-0.4909,1.0116);
\draw (0.1283,0.9679) -- (0.2012,1.1791);
\draw (0.6091,0.7165) -- (0.8385,0.8185);
\draw (0.8568,0.3305) -- (1.0826,0.3887);
\draw (0.8859,-0.0228) -- (1.1627,-0.0338);
\draw (0.8677,-0.2741) -- (1.1518,-0.3324);
\draw (-1.2485,0.5017) -- (-0.9716,0.4507);
\draw (-1.2084,-0.2195) -- (-0.9097,-0.1394);
\draw (-0.8806,1.1245) -- (-0.7167,0.9242);
\draw (-0.8551,-0.7986) -- (-0.6366,-0.5692);
\draw (-0.1449,1.252) -- (-0.1631,1.0225);
\draw (0.5399,1.0298) -- (0.4051,0.8513);
\draw (0.8167,-0.5437) -- (1.0935,-0.6311);

\draw (0.9879,-0.9115) -- (0.7256,-0.8095);
\draw (0.6164,-1.0208) -- (0.8531,-1.1701);
\draw (0.6746,-1.3741) -- (0.4853,-1.2138);
\draw  (0.5071,-1.4652) -- (0.5071,-1.4652);
\draw (0.478,-1.8148) -- (0.6455,-1.7055);
\draw (0.8495,-2.1863) -- (1.0899,-2.0552);
\draw (1.0389,-2.5505) -- (1.2902,-2.5032);
%\draw (1.1627,-2.8783) -- (1.3667,-2.8383);
\draw (0.2708,-1.4511) -- (0.5007,-1.4559);

\draw (-0.316,-2.6962) -- (-0.0137,-2.6926);

\draw (-0.3998,-1.1264) -- (-0.2505,-0.9006);

\path[red, draw,line width=0.6pt] (-0.1376,-3.024) -- (-0.174,-2.6889) -- (-0.1558,-2.3029) -- (0.1065,-1.8585) -- (0.1392,-1.4579) -- (-0.1048,-1.1701) -- (-0.316,-1.0281) -- (-0.5346,-0.8715) -- (-0.7422,-0.6821) -- (-0.917,-0.4708) -- (-1.0591,-0.1831) -- (-1.1283,0.1119) -- (-1.121,0.4798) -- (-1.0044,0.8149) -- (-0.8077,1.0225) -- (-0.5237,1.1318) -- (-0.1485,1.1427) -- (0.1684,1.0771) -- (0.4816,0.946) -- (0.7184,0.7566) -- (0.977,0.356) -- (1.0279,-0.0338) -- (1.0061,-0.3033) -- (0.9369,-0.5837) -- (0.864,-0.8569) -- (0.7293,-1.09) -- (0.569,-1.294) -- (0.4088,-1.4615) -- (0.5581,-1.7565) -- (0.966,-2.1171) -- (1.1663,-2.5141) -- (1.272,-2.8528);

\fill (0.2708,-1.4511) circle (1pt);
\draw (0.2508,-1.4011) node[above] {$v$};

\draw (0.07,-1.3189) node {\footnotesize $\mathfrak{c}_1$};
\draw (0.1352,-1.6551) node {\footnotesize $\mathfrak{c}_2$};

\draw (0.6458,-3.5) node {a)};

%czesc b)

\begin{scope}[shift={(6,0)}]

\draw  plot[smooth, tension=.7] coordinates {(-1.3298,-2.7993) (-1.4445,-2.0291) (-1.902,-1.4228) (-2.3745,-1.0103) (-2.7725,-0.553) (-2.9469,0.2323) (-2.7189,0.7642) (-2.1032,1.1471) (-1.1491,1.053) (-0.4336,0.6321) (-0.1504,-0.1878) (-0.2846,-0.9182) (-0.5756,-1.4154) (-0.5976,-2.1405) (-0.5678,-2.7219)};
\draw  plot[smooth, tension=.7] coordinates {(-1.7156,-2.8859) (-1.7752,-2.1405) (-2.133,-1.6933) (-2.6845,-1.3356) (-3.2495,-0.4916) (-3.4254,0.2445) (-3.0931,1.0793) (-2.143,1.5138) (-0.908,1.4256) (-0.1336,0.8568) (0.254,-0.2155) (0.1347,-1.1099) (-0.1639,-1.6213) (-0.1932,-2.7496)};
\draw (-1.7736,-2.1488) -- (-1.4516,-2.0432);
\draw (-2.143,-1.6843) -- (-1.8949,-1.4416);
\draw (-2.676,-1.3413) -- (-2.3857,-1.0035);
\draw (-3.0454,-0.8663) -- (-2.7182,-0.6605);
\draw (-3.2882,-0.3966) -- (-2.8871,-0.2277);
\draw (-3.4149,0.1839) -- (-2.9399,0.1839);
\draw (-3.3515,0.6536) -- (-2.8977,0.5111);
\draw (-3.0824,1.06) -- (-2.7024,0.7592);
\draw (-2.6496,1.3605) -- (-2.4385,1.0178);
\draw (-1.9899,1.1339) -- (-2.0269,1.5452);
\draw (-1.515,1.1283) -- (-1.4727,1.5241);
\draw (-1.0611,1.0122) -- (-0.9661,1.4239);
\draw (-0.7603,0.8961) -- (-0.5386,1.255);
\draw (-0.4806,0.6692) -- (-0.2061,0.9436);
\draw (-0.3064,0.4159) -- (0.0261,0.5689);
\draw (-0.185,0.0411) -- (0.1897,0.115);
\draw (-0.1428,-0.3177) -- (0.2688,-0.3336);
\draw (-0.2378,-0.8033) -- (0.1844,-0.951);
\draw (-0.47,-1.2149) -- (-0.0795,-1.3943);
\draw (-0.6125,-1.7374) -- (-0.185,-1.8271);
\draw (-0.5967,-2.2282) -- (-0.2009,-2.2387);
\draw (-1.7208,-2.5132) -- (-1.3619,-2.4234) node (v1) {};
\draw (v1);

\fill (-1.3582,-2.4155) circle (1pt);
\fill (-0.5923,-2.2366) circle (1pt);
\path[draw,line width=0.6pt] (-1.3582,-2.4155) -- (-0.5923,-2.2366);

%hipergraf

\path[red, draw,line width=0.6pt] (-1.5063,-2.8599) -- (-1.537,-2.4733) -- (-1.6229,-2.0867) -- (-2.034,-1.5836) -- (-2.5372,-1.1786) -- (-2.8624,-0.7613) -- (-3.1017,-0.3011) -- (-3.2183,0.1837) -- (-3.1385,0.5764) -- (-2.9115,0.92) -- (-2.5495,1.2023) -- (-2.0033,1.3434) -- (-1.494,1.3311) -- (-1.0093,1.2452) -- (-0.6472,1.0857) -- (-0.3466,0.8096) -- (-0.1502,0.4905) -- (-0.0091,0.0855) -- (0.0339,-0.3256) -- (-0.0459,-0.8718) -- (-0.2852,-1.3013) -- (-0.4141,-1.7677) -- (-0.4079,-2.2279) -- (-0.4018,-2.731);

\draw (-1,-3.5) node {b)};
\draw (-1,-2) node[below] {$e$};

\end{scope} 

\end{tikzpicture}
\caption{Two types of hypergraphs osculations.}
\label{Figure:Osculating}
\end{figure}
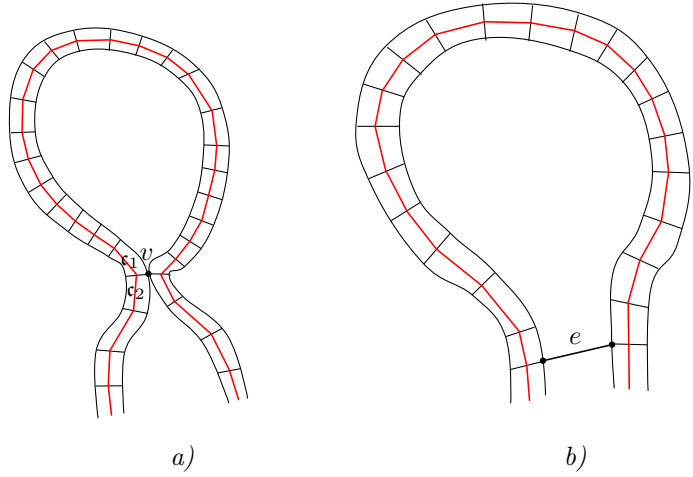
\end{lem}

\begin{proof}
(1) Suppose, on the contrary, that such a point $v$ exists. Then let $e_1, e_2$ be the two edges, ending in $v$ that are dual to the hypergraph segment $\Lambda'$ and that are none of them is an internal edge of a divided tile. Let $\mathcal{D}$ be a diagram collared by $\Lambda'$ and an edge-path $\{e_1, e_2\}$. Such diagram satisfies $|\tilde{\partial} \mathcal{D}| \leq |\mathcal{D}| + 2$. Therefore, by Theorem \ref{Theorem:Generalized_Isoperimetric_Inequality} we obtain $|\mathcal{D}| \leq 3$. Hence $|\tilde{\partial} \mathcal{D}| \leq 5$, and combining this with Lemma
 \ref{Lemma:Pairity_of_the_boundary} we get that $|\tilde{\partial} \mathcal{D}| \leq 4$. This by Theorem \ref{Theorem:Generalized_Isoperimetric_Inequality} gives us that that $|\mathcal{D}| \leq 2$. A diagram consisting of two 2-cells and having the generalized boundary length at most 4 must be a pair of strongly adjacent 2-cells. This ends the proof in the case c). If $\Lambda'$ is red or blue hypergraph we need to provide further argument. Note, that both red and blue hypergraph join the antipodal midpoints of bundary edges of every tile. Hence, there are no two edges that are dual do $\Lambda'$, have common end and non of them is an internal edge of a divided tile. This ends the proof in cases a) and b).

Therefore there is a blue 2-cell and a red 2-cell in the carrier of $\Lambda'$, which contradicts assumptions a), b) and c).

(2) Let $v_1$ and $v_2$ be the ends of the edge $e$. For $i=1,2$ let $e_i$ be the edge dual to $\Lambda'$ ending in $v_i$ that is not an internal edge of a divided tile. Let $\mathcal{E}$ be a diagram with horns collared by $\Lambda'$ and the edge-path $\{e_1, e, e_2 \}$. As previously, $\mathcal{E}$ satisfies $|\partial \mathcal{E}| \leq |\mathcal{E}| + 2$. Therefore by \ref{Theorem:Generalized_Isoperimetric_Inequality} $|\mathcal{E}| \leq 3$, so $|\partial \mathcal{E}| \leq 5$. By Lemma \ref{Lemma:Pairity_of_the_boundary} we obtain that in fact $|\partial \mathcal{E}| \leq 4$, which, by previous observation, means that $|\mathcal{E}| \leq 2$. A diagram with two 2-cells and the boundary length equal 4 must be a pair of strongly adjacent 2-cells. In a diagram which is a pair of strongly adjacent 2-cells there is no edge which does not lie in $\text{Car}(\Lambda)$ but its endpoints do. This end the proof in cases a), b) and c).
\end{proof}

\begin{lem}\label{Lemma:Intersecting_intersecting_geodesic}
Let $\Lambda_1$ be a hypergraph segment and $\gamma$ be a geodesic segment in $\cay^{(1)}$ that intersects $\Lambda_1$ in two points, which are the endpoints of $\gamma$. Let $\Lambda_2 \neq \Lambda_1$ be a hypergraph in $\cay$ intersecting $\gamma$ in point $x$.
\begin{enumerate}
\item $\Lambda_1$ and $\Lambda_2$ pass through a common 2-cell $\mathfrak{c}_x$, that is $\text{Car}(\Lambda_1) \cap \text{Car}(\Lambda_2) \supseteq \mathfrak{c}_x$.
\item If moreover $\Lambda_2$ intersects $\gamma$ in two different points $x$ and $y$ that do not lie in the boundary of the same 2-cell of $\text{Car}(\Lambda_1)$ then there are two different 2-cells in $\text{Car}(\Lambda_1) \cap \text{Car}(\Lambda_2)$.
\end{enumerate}
\end{lem}

\begin{proof}
Part (1) Let $e$ be an edge of $\gamma$ dual to $\Lambda_2$. There are two possibilities
\begin{itemize}
  \item The edge $e$ lies in the boundary of a 2-cell $\mathfrak{c} \in \text{Car}(\Lambda_1)$. In that case $\Lambda_2$ is dual to the boundary edge of $\mathfrak{c}$, so $\mathfrak{c}_x := \mathfrak{c}$ belongs to $\text{Car}(\Lambda_2)$. See Figure \ref{Figure:Intersecting_intersecting_geodesic} a).
  \item The edge $e$ does not lie in the boundary edge of a 2-cell $\mathfrak{c} \in \text{Car}(\Lambda_1)$. Since $\gamma$ is a geodesic edge-path, by Lemma \ref{Lemma:Geodesic_breaks_away_from_carrier} combined with Lemma \ref{Lemma:Osculating} we know that $e$ lies in the boundary of a house diagram presented in Figure \ref{Figure:Intersecting_intersecting_geodesic} b). It is clear from the Figure \ref{Figure:Intersecting_intersecting_geodesic} b)  that $\mathfrak{c}_x \in \text{Car}(\Lambda_1) \cap \text{Car}(\Lambda_2)$.
\end{itemize}

Part (2). If $x$ and $y$ lie in the boundaries of distinct 2-cells $\mathfrak{c}_x$ and $\mathfrak{c}_y$ belonging to $\text{Car}(\Lambda)$ then it is clear that $\mathfrak{c}_x$ and $\mathfrak{c}_y$ both lie in $\text{Car}(\Lambda_1) \cap \text{Car}(\Lambda_2)$ (see Figure \ref{Figure:Intersecting_intersecting_geodesic}.

Consider the case where $x$ lies in the boundary of a 2-cell $\mathfrak{c}_1$ in $\text{Car}(\Lambda_1)$ and $y$ do not. Considering a hypergraph $\Lambda_1$ dual to $\gamma$ passing through $y$ we will find (as in part (1) ) a regular 2-cell $\mathfrak{c}_2$ that lies in $\text{Car}(\Lambda_1) \cap \text{Car}(\Lambda_2)$ and that there is no edge of $\gamma$ in the boundary of $\mathfrak{c}_2$. Hence $\mathfrak{c}_2 \neq \mathfrak{c}_1$ and $\mathfrak{c}_1$ and $\mathfrak{c}_2$ both lie in $\text{Car}(\Lambda_1) \cap \text{Car}(\Lambda_2)$.

The last case to consider is when neither $x$ nor $y$ lie in the boundary of $\text{Car}(\Lambda_1)$. First, note that $x$ and $y$ cannot lie in the roof of the same house diagram since then there would be a self intersection of $\Lambda_2$ (see Figure \ref{Figure:Intersecting_intersecting_geodesic} b) ) which contardicts Theorem \ref{Theorem:Red_hypergraphs_are_trees}. Hence, $x$ and $y$ lie in the roofs of different house diagrams, which means that one 2-cell from each of those diagrams belongs to $\text{Car}(\Lambda_1) \cap \text{Car}(\Lambda_2)$. This means that there are at least two different 2-cells in $\text{Car}(\Lambda_1) \cap \text{Car}(\Lambda_2)$.

%If $x$ and $y$ are the middles of two consecutive edges $e_x, e_y$ of $\gamma$ that are not  in the boundary of $\text{Car}(\Lambda_1)$ then by Lemma \ref{Lemma:Geodesic_breaks_away_from_carrier} combined with Lemma \ref{Lemma:Osculating} $e_x$ and $e_y$ lie in the roof of the same house diagram. Hence the hypergraphs dual to $e_x$ and $e_y$ pass through different 2-cells of $\textit{Car}(\Lambda_1)$, see Figure \ref{Figure:Intersecting_intersecting_geodesic} b). In that case we are done. If this is not the case the points $x$ and $y$ lie in the boundaries of a different house diagrams. In that case we easily see that there is a self intersection of hypergraph $\Lambda_2$ which is not allowed according to \ref{Theorem:Red_hypergraphs_are_trees} and \cite[Lemma 5.16]{odrz}.

%there are at least two different 2-cells in $\text{Car}(\Lambda_1) \cap \text{Car}(\Lambda_2)$, since there is at least one such a 2-cell in every of mentioned house diagrams.

\begin{figure}[h!]
\centering
\begin{tikzpicture}[scale=0.9]

%czesc a)

%2 komorki:

\filldraw[draw=black,fill=gray!20] (-5,-1) rectangle (-3,1);

%hipergraf
\path[red, draw,line width=2pt] (-5,0) -- (-3,0);
\draw (-3.3,-0.25) node { ${\color{red} \Lambda_1}$};

%\path[blue, draw,line width=2pt] (-4,1) -- (-4,-1);
%\draw (-3.7,0.7) node { ${\color{blue} \Lambda_2}$};

%geodezyjna
\draw[draw,line width=1.5pt, dashed] (-6,-0.5) -- (-5,-1) -- (-3,-1)--(-2,-0.5);

%krawedz
\path[black, draw,line width=2pt] (-5,-1) -- (-3,-1);

%wezly
\draw (-5, 1) node[below right] {$\mathfrak{c}_x$};
\draw (-4.7,-1) node[below] {$e$};
\fill (-4,-1) circle (2.5pt);
\draw (-4,-1) node[below] {$x$};
\draw (-5.4,-0.75) node[below left] {$\gamma$};
\draw (-4,-2.5) node {a)};

%czesc b)

%2 komorki:
\fill[fill=gray!20] (0,-1)--(2,-2.5)--(4,-1)--cycle;
\filldraw[draw=black,fill=gray!20] (0,-1) rectangle (2,1);
\filldraw[draw=black,fill=gray!20] (2,-1) rectangle (4,1);

%hipergraf
\path[red, draw,line width=2pt] (0,0) -- (4,0);
\draw (3.5,-0.25) node { ${\color{red} \Lambda_1}$};

\path[blue, draw,line width=2pt] (1,1) -- (1,-1) -- (3, -1.75);
\draw (1.3,0.7) node { ${\color{blue} \Lambda_2}$};

\path[blue, draw,line width=2pt] (3,1) -- (3,-1) -- (1, -1.75);
%\draw (1.3,0.7) node { ${\color{blue} \Lambda_2}$};

%geodezyjna
\draw[draw,line width=1.5pt, dashed] (-1,-1) -- (0,-1) -- (2,-2.5) -- (4,-1) -- (5,-1.5);

%krawedz
\path[black, draw,line width=2pt] (2,-2.5) -- (4,-1);

%wezly
\draw (0, 1) node[below right] {$\mathfrak{c}_x$};
\draw (2.2,-2.25) node[below right] {$e$};
\fill (3,-1.75) circle (2.5pt) node[below right] {$x$};
\fill (1,-1.75) circle (2.5pt) node[below left] {$y$};
\draw (-0.5,0) node[below] {$\gamma$};
\draw (1,-3) node {b)};

\end{tikzpicture}
\caption{Intersections of hypergraphs $\Lambda_1$ and $\Lambda_2$.}
\label{Figure:Intersecting_intersecting_geodesic}
\end{figure}
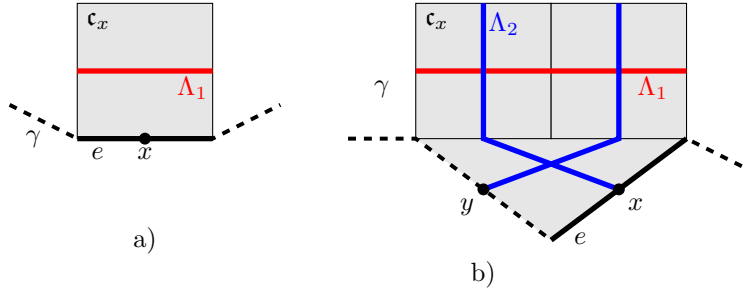
\end{proof}

%The construction of finding a 2-cell $\mathfrak{c}_x$ described in part (1) of the proof provides an injective map from the set of midpoints of edges of $\gamma$ to the set of 2-cells of $\text{Car}(\Lambda)$. The injectivity results simply by a construction of the procedure. Hence two distinct point of intersection of $\gamma$ and $\Lambda_2$ gives us two distinct 2-cells in $\text{Car}(\Lambda_1) \cap \text{Car}(\Lambda_2)$.

\begin{lem}\label{Lemma:There_is_a_color_cell_in_diagram}
Let $\Lambda$ be a hypergraph in $\cay$. Let $\gamma$ be a geodesic edge-path in the 1-skeleton of $\cay$. Suppose that $\gamma$ intersects $\Lambda$ in two consecutive points $x$ and $y$. Let $\Lambda_{xy}$ be the segment of $\Lambda$ joining $x$ and $y$. Then:
\begin{enumerate}
\item if $\Lambda$ is a red hypergraph, there is an isolated red 2-cell $\mathfrak{c}$ of $\Lambda_{xy}$,
\item if $\Lambda$ is a blue hypergraph, there is an isolated blue 2-cell $\mathfrak{c}$ of $\Lambda_{xy}$,
\item if $\Lambda$ is a standard hypergraph, there is a divided tile of $\Lambda_{xy}$, such that the none of its internal edges belong to $\gamma$.
\end{enumerate}
\end{lem}

\begin{proof}
The proof relies on the idea, that if a geodesic crosses a hypergraph two times it must be worthwile in terms of distance. If the hypergraph does not ``turn'' there is no reason for a geodesic to cross it two times. A colored hypergraph can ``turn'' only if there is an isolated 2-cell of the segment of hypergraph and a standard hypergraph ``effectively turns'' when meeting a pair of strongly adjacent 2-cells. Let now state a precise argument.

First, consider the case where $\Lambda$ is a red hypergraph. Suppose, on the contrary, that there is no isolated red 2-cell of $\Lambda_{xy}$. The the carrier of $\Lambda_{xy}$ consists only of tiles (regular 2-cells or divided tiles) and every edge of $\Lambda_{xy}$ joins an antipodal midpoints of boundary edges of every tile in $\text{Car}(\Lambda')$ i.e. the carrier of $\Lambda'$ has the shape illustrated in Figure \ref{Figure:Straight_segment}.

\begin{figure}[h]
\centering
\begin{tikzpicture}[scale=0.9]

%2 komorki:

\fill[fill=gray!20] (0,1) -- (1,2) -- (2,1) -- (3,1) -- (4,2) -- (4.5,1.5)--(4.5,1)--cycle;
\fill[fill=gray!20] (6,1) -- (7,2) -- (8,1) -- (9,2) -- (10,1)--cycle;

\fill[fill=gray!20] (4,0) rectangle (4.5,1);
\fill[fill=gray!20] (5.5,0) rectangle (6,1);

\filldraw[draw=black,fill=gray!20] (0,0) rectangle (1,1);
\filldraw[draw=black,fill=gray!20] (1,0) rectangle (2,1);
\filldraw[draw=black,fill=gray!20] (2,0) rectangle (3,1);
\filldraw[draw=black,fill=gray!20] (3,0) rectangle (4,1);

\filldraw[draw=black,fill=gray!20] (6,0) rectangle (7,1);
\filldraw[draw=black,fill=gray!20] (7,0) rectangle (8,1);
\filldraw[draw=black,fill=gray!20] (8,0) rectangle (9,1);
\filldraw[draw=black,fill=gray!20] (9,0) rectangle (10,1);

\draw (4,0)-- (4.5,0);
\draw (4,1)-- (4.5,1);
\draw (5.5,0)-- (6,0);
\draw (5.5,1)-- (6,1);

%hipergraf

\path[red, draw,line width=2pt] (0,0.5) -- (4.5,0.5);
\path[red, draw,line width=2pt] (5.5,0.5) -- (10,0.5);

\draw (2.5,0.25) node { ${\color{red} \Lambda'}$};

\draw (5,0.5) node {$\dots$};

%klamerka

\draw [decoration={brace, amplitude=10pt, mirror}, decorate] (0,-0.2) -- (10,-0.2);
\draw (5,-0.8) node {$k$};

%geodezyjna
\draw[draw,line width=1.5pt, dashed] (0,0) -- (0,1) -- (1,2) -- (2,1) -- (3,1) -- (4,2) -- (4.5,1.5);
\draw[draw,line width=1.5pt, dashed] (5.5,1) -- (6,1) -- (7,2) -- (8,1) -- (9,2) -- (10,1)--(10,0);

%wezly
\fill (0,0.5) circle (1.5pt) node[left] {$x$};
\fill (10,0.5) circle (1.5pt) node[right] {$y$};
\draw (0.5,1.5) node[above left] {$\gamma$};

\fill (0,1) circle (1.5pt) node[left] {$x'$};
\fill (10,1) circle (1.5pt) node[right] {$y'$};

\fill (0,0) circle (1.5pt) node[left] {$x''$};
\fill (10,0) circle (1.5pt) node[right] {$y''$};

\end{tikzpicture}
\caption{The diagram collared by $\Lambda_{xy}$ and $\gamma$.}
\label{Figure:Straight_segment}
\end{figure}
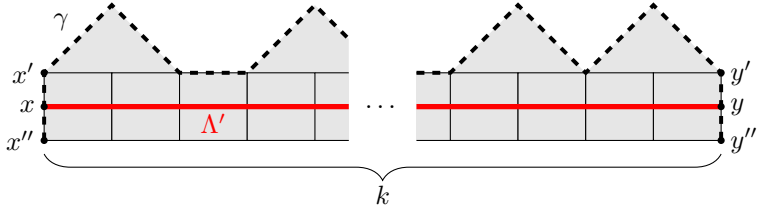

Let $x'$ and $x''$ be the two ends of the edge of $\cay$ containing $x$, and suppose that $x'$ lies on $\gamma$ between $x$ and $y$. Let $y'$ and $y''$ be the two ends of the edge of $\cay$ containing $y$, and suppose that $y'$ lies on $\gamma$ between $x$ and $y$. Let $\gamma'_0$ be the subpath of $\gamma$ joining $x'$ and $y'$ (see Figure \ref{Figure:Straight_segment}). Let $V$ be the set of vertices of $\gamma'$. Define
%$$ V' = \{v \in V | \text{there is exactly one edge of }\gamma'  \text{ ending in }v \text{ and belonging to } \text{Car}(\Lambda') \}.$$
$$ V' = \{v \in V | \text{there is at least one edge of }\gamma'  \text{ ending in }v \text{ and not belonging to } \text{Car}(\Lambda') \}.$$

For each pair ${s,t}$ of two consecutive vertices in $V'$ consider the subsegment of $\gamma'$ joining $s$ and $t$ and denote it by $\gamma'_{st}$.  By Lemma \ref{Lemma:Osculating} part (2) we obtain that $\gamma'_{st}$ either
\begin{enumerate}
\item lies in the boundary of the carrier of $\Lambda$, or
\item has a length at least 2.
\end{enumerate}

In the latter case, we know that $\gamma'_{st}$ is a roof of a house diagram. Therefore without change of the total length of $\gamma'$ we can replace $\gamma'_{st}$ with an edge-path of length 2, that lies in the boundary of $\text{Car}(\Lambda_{xy})$ and does not cross $\Lambda_{xy}$.

Hence, we can modify $\gamma'$  to a path having the same length such that all its edges lie in the boundary of $\text{Car}(\Lambda_{xy})$. From now $\gamma'$ denotes modified path. For every vertex $v$ of $\gamma'$ denote by $e_v$ the edge dual to $\Lambda'$  ending in $v$ that is not an internal edge of a divided tile (this definition is unambiguous according to Lemma \ref{Lemma:Osculating} part (1) ). We will prove now that the length $d$ of $\gamma'$ equals the number $k$ of tiles of $\Lambda_{xy}$. Every edge of $\gamma'$ lies in the boundary of exactly one tile so $k \geq d$. Suppose, on the contrary, that $k > d$. This means that there must be a vertex of $\gamma'$ that is a common vertex of more than two tiles. But in that case there are at least two edges dual to $\Lambda_{xy}$ that are not internal edges of a divided tile. This contradicts Lemma \ref{Lemma:Osculating} part 1). Hence $k = d$. Consider now the edge-path $\gamma''$ in $\cay$ joining $x''$ and $y''$ and goes along the boundary of $\text{Car}(\Lambda_{xy})$. The length of this edge-path is not greater then $k$ since every tile of $\Lambda_{xy}$ shares at most one boundary edge with $\gamma''$. Hence there exists an edge-path joining $x''$ and $y''$ of length $k$. Note that the segment of $\gamma$ joining $x''$ and $y''$ has length $k+2$. This contradicts the fact that $\gamma$ is a geodesic.

If $\Lambda'$ is a blue hypergraph the proof is analogous.

Consider now the case, where $\Lambda$ is a standard hypergraph. In that case we can provide an analogous argument to prove that there must be a divided tile of $\Lambda_{xy}$.  We only need to prove that none of the internal edges of a divided tile belong to $\gamma$. Suppose, on the contrary, that this is the case.

%By Lemma \ref{Lemma:Diagram_collared_by_path_and_hypergraph} there exists a planar diagram $\mathcal{D}$ collared by $\Lambda_{xy}$ and $\gamma$. So if $\gamma$ contains an internal edge  that is not dual to $\Lambda_{xy}$ it also contains the second one due to the fact that $\mathcal{D}$ is planar and $\gamma$ can not end before it intersects $\Lambda'$ and does not have backtracks. Hence $\gamma$ contains a distinguished edge dual to $\Lambda_{xy}$. But this means that $\gamma$ intersects $\Lambda$ between points $x$ and $y$, which is a contradiction.
By Lemma \ref{Lemma:Diagram_collared_by_path_and_hypergraph} there exists a planar diagram $\mathcal{D}$ collared by $\Lambda_{xy}$ and $\gamma$. By definiton of a diagram collared by a hypergraph and path $\gamma$ cannot contain any internal edge of $\mathcal{D}$ which ends the proof. 
\end{proof}

The definition of a 2-collared diagram in case of standard hypergraphs can be found in \cite[Subsection 3.3]{ow11}. In a nutshell this is a diagram collared by two hypergraphs (see Definition \cite[Lemma 3.17]{ow11}). Let us state this definition for a case of general hypergraphs.

\begin{defi}[2-collared diagram]\label{Definition:2collared}
Let $\lambda_1$ and $\lambda_2$ be two hypergraph segments. Let $D$ be a van Kampen diagram with the following additional properties:
\begin{enumerate}
  \item There are two external 2-cells of $D$ called \textit{corners}
  \item $\lambda_1$ and $\lambda_2$ pass only through external 2-cells of $D$
  \item every external edge of $D$ belongs to $\text{Car}(\lambda_1) \cup \text{Car}(\lambda_2)$
  \item corners are the only 2-cells of $D$ containing the edges of both hypergraph segments $\lambda_1$ and $\lambda_2$.
\end{enumerate}
We call $D$ \textit{2-collared diagram (collared by $\lambda_1$ and $\lambda_2$)}.
\end{defi}

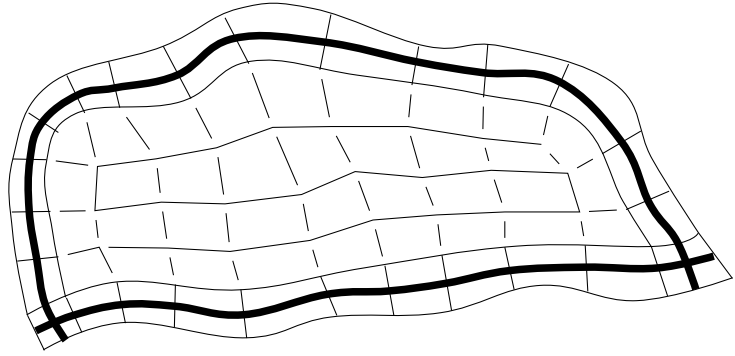
\begin{figure}[h]
\centering
\begin{tikzpicture}[scale=0.5]

\draw  plot[smooth, tension=.7] coordinates {(-22.1855,7.9224) (-20.0447,8.6486) (-16.7913,8.2419) (-14.5545,8.7938) (-11.6205,8.91) (-8.977,9.6363) (-6.4207,9.2296) (-3.8644,9.8396)};
\draw  plot[smooth, tension=.7] coordinates {(-22.6301,8.8519) (-20.286,9.7234) (-16.9454,9.5201) (-13.6629,10.1011) (-9.596,10.682) (-5.8196,10.682) (-4.7738,11.0306)};

\draw (-22.6089,8.8386) node (v1) {} -- (-22.1606,7.9075);
\draw (-20.2294,9.718) node (v5) {} -- (-20.0053,8.6317);
\draw (-16.9361,9.5283) node (v21) {} -- (-16.7809,8.2351);
\draw (-14.8153,9.8731) node (v22) {} -- (-14.3842,8.8386);
\draw (-11.3323,8.9593) -- (-11.5909,10.4421) node (v32) {};
\draw (-9.6698,9.5223) -- (-9.9185,10.6528) node (v37) {};
\draw (-7.7114,9.4766) -- (-7.7803,10.7008) node (v42) {};
\draw (-5.5905,9.3559) -- (-6.0216,10.6663) node (v3) {};
\draw (-4.7629,11.0284) node (v2) {} -- (-3.8663,9.8214);
\draw (-12.8695,8.8449) -- (-13.106,10.1589) node (v27) {};
\draw (-18.7037,8.5296) -- (-18.6774,9.6596) node (v7) {};
\draw  plot[smooth, tension=.7] coordinates {(v1) (-23,11) (-23,13) (-22,15) (-19,16) (-17,17) (-15,17) (-12,16) (-10,16) (-7,15) (-6,13) (v2)};
\draw  plot[smooth, tension=.7] coordinates {(-21.5945,9.318) (-21.9362,11.0262) (-22.1202,12.8395) (-21.3317,14.2061) (-18.546,14.5215) (-16.6801,15.599) (-13.8419,15.231) (-10.6094,14.7317) (-7.9814,14.0221) (-6.7725,11.8672) (v3)};
\draw (-21.575,9.3043) -- (-21.1671,8.3904);
\draw (-22.8805,10.2833) -- (-21.7872,10.3976) node (v4) {};
\draw (-23.0273,11.5888) -- (-21.9993,11.6051) node (v13) {};
\draw (-23,13) -- (-22.1054,12.9816) node (v14) {};
\draw (-22.5802,14.2252) -- (-21.7888,13.7052);

\draw (-20.1554,14.3628) node (v11) {} -- (-20.4491,15.5867);
\draw (-19,16) -- (-18.2625,14.6076) node (v17) {};
\draw (-16.7241,15.5366) node (v26) {} -- (-17.3572,16.7575);

\draw (-14.8357,15.3745) node (v31) {} -- (-14.542,16.8432);

\draw (-12.388,14.9666) node (v36) {} -- (-12.2085,15.9946);
\draw (-10.4837,14.6432) node (v41) {} -- (-10.3707,16.0451);
\draw (-8.7201,14.3945) node (v48) {} -- (-8.2227,15.525);

\draw (-7.3294,13.139) node (v45) {} -- (-6.3014,13.7428);
\draw (-6.693,11.654) node (v44) {} -- (-5.5508,12.1436);
\draw (v4) -- (-20.7035,10.6528) node (v6) {} -- (v5);
\draw (v6) -- (-18.8721,10.6302) node (v8) {} -- (-17.2216,10.5397) node (v20) {} -- (-15.0962,10.8337) node (v23) {} -- (-13.423,11.3763) node (v28) {} -- (-11.7273,11.512) node (v33) {} -- (-9.9185,11.5909) node (v38) {} -- (-7.9288,11.5909) node (v43) {} -- (-8.2453,12.6199) node (v46) {} -- (-10.2802,12.6988) node (v39) {} -- (-12.1117,12.5068) node (v34) {} -- (-13.8979,12.6651) node (v29) {} -- (-15.3223,12.0546) node (v24) {} -- (-17.3572,11.8059) node (v19) {} -- (-19.053,11.8511) node (v9) {} -- (-20.8166,11.625) node (v12) {} -- (-20.7487,12.8008) node (v15) {} -- (-19.1886,13.0043) node (v10) {} -- (-17.5833,13.2982) node (v18) {} -- (-16.091,13.8408) node (v25) {} -- (-14.5309,13.8634) node (v30) {} -- (-12.496,13.8634) node (v35) {} -- (-10.5063,13.5469) node (v40) {} -- (-8.9462,13.3886) node (v47) {};
\draw (v7) -- (v8) -- (v9) -- (v10) -- (v11);
\draw (v6) -- (v12);
\draw (v13) -- (v12);
\draw (v14);
\draw (v15) -- (v14);
\draw (-21.0879,14.2252) node (v16) {} -- (-21.5627,15.2201);
\draw (v15) -- (v16);
\draw (v17) -- (v18) -- (v19) -- (v20) -- (v21);
\draw (v22) -- (v23) -- (v24) -- (v25) -- (v26);
\draw (v27) -- (v28) -- (v29) -- (v30) -- (v31);
\draw (v32) -- (v33) -- (v34) -- (v35) -- (v36);
\draw (v37) -- (v38) -- (v39) -- (v40) -- (v41);
\draw (v42) -- (v43) -- (v44);
\draw (v45) -- (v46) -- (v47) -- (v48);

\draw[line width=1mm]  plot[smooth, tension=.7] coordinates {(-21.5935,8.1767) (-22.1361,9.1264) (-22.3622,10.3247) (-22.5657,11.6361) (-22.5431,12.9927) (-22.2265,13.9649) (-21.2543,14.7111) (-20.2821,14.892) (-18.6315,15.2311) (-17.1392,16.2033) (-14.7426,16.1129) (-12.3007,15.5929) (-10.5145,15.2989) (-8.5022,15.0728) (-6.8064,13.4449) (-6.0603,11.7944) (-5.3367,10.7995) (-4.8393,9.5333)};
\draw[line width=1mm]  plot[smooth, tension=.7] coordinates {(-22.3848,8.4254) (-21.4126,8.8324) (-20.0786,9.1264) (-18.6767,9.0811) (-16.8679,8.855) (-14.6521,9.3977) (-13.0016,9.4881) (-11.4867,9.6916) (-9.7909,10.0308) (-7.8012,10.1212) (-5.8794,10.0986) (-4.3645,10.4151)};

\end{tikzpicture}
\caption{2-collared diagram. Collaring segments are indicated with thick lines.}
\label{Figure:2_collared_diagram}
\end{figure}

\begin{lem}[reformulation of {\cite[Theorem 3.12]{ow11}}]\label{Lemma:2collared_diagram}
Suppose, $\Lambda_1$ and $\Lambda_2$ are distinct hypergraphs that are embedded trees in $\cay$, and that they intersect at least twice. Let $\mathfrak{c}$ be a common 2-cell of $\text{Car}(\Lambda_1)$ and $\text{Car}(\Lambda_2)$  Then there exists a 2-collared diagram collared by segments $\Lambda_1$ and $\Lambda_2$. Moreover we can choose $\mathfrak{c}$ to be one of its corners.
\end{lem}

There is a definition of a 2-collared diagram in the beginning of \cite[Section 3.11]{ow11}. However this is the special case of the general definition of a diagram collared by hypergraphs and paths that was stated for a general polygonal complex and for the standard hypergraphs. Definition \ref{Definition:2collared} is a modification of it: we use it for $l=4$ and for general hypergraphs, but the fact that the hypergraph joins the antipodal midpoints of edges in $\cay$ is not necessary to provide the proof of Lemma \ref{Lemma:2collared_diagram}, in fact the proof can be provided, as well, in the case of colored hypergraphs.

\begin{lem}\label{Lemma:Shape_of_2_collared}
W.o.p. every 2-collared diagram in $\cay$ is a pair of strongly adjacent 2-cells.
\end{lem}

\begin{proof}
Let $\mathcal{D}$ be a 2-collared diagram (with horns) collared by hypergraphs $\Lambda_1$ and $\Lambda_2$. Let $k_1$ be the length of the collaring segment of $\Lambda_1$ and $k_2$ be the length of the collaring segment of $\Lambda_2$. Let $h$ denote the number of horns in $\mathcal{D}$. Each pair of strongly adjacent 2-cells gives a contribution to the generalized boundary length equal 2, hence $|\tilde{\partial} \mathcal{D}| \leq k_1+k_2+h$. There are two 2-cells in $\emph{Car}(\Lambda_1) \cap \emph{Car}(\Lambda_2)$, so $|\mathcal{D}| \geq k_1+k_2+h-2$. Hence, by Theorem \ref{Theorem:Generalized_Isoperimetric_Inequality} for the density $d < \frac{3}{10}$  we obtain that w.o.p.:

\begin{equation}\label{eq:intersection}
k_1+k_2+h \geq |\tilde{\partial} \mathcal{D}| > \frac{8}{5} |\mathcal{D}| > \frac{8}{5}(k_1 + k_2 + h - 2).
\end{equation}

Hence $k_1+k_2+h \leq 5$. By Lemma \ref{Lemma:Pairity_of_the_boundary} we know that $|\tilde{\partial} \mathcal{D}|$ is even. Therefore $|\tilde{\partial} \mathcal{D}| \leq 4$ and by (\ref{eq:intersection}) we obtain $|\mathcal{D}| \leq 2$. Hence, $\mathcal{D}$ must be a diagram consisting of two strongly adjacent 2-cells.
\end{proof}

\begin{lem}\label{Lemma:Regular_Cell_Intersection_Once}
The following statement holds w.o.p.: Let $\Lambda_1$ and $\Lambda_2$ be two hypergraphs in $\cay$. Suppose that $\Lambda_1$ and $\Lambda_2$ intersect transversally in a 2-cell $\mathfrak{c}$, meaning that $\emph{Car}(\Lambda_1) \cap \emph{Car}(\Lambda_2)$ contains a regular 2-cell $\mathfrak{c}$ and that $\Lambda_1 \cap \mathfrak{c} \neq \Lambda_2 \cap \mathfrak{c}$. Then $\Lambda$ and $\Lambda'$ intersect only once. Moreover $\mathfrak{c}$ is the only one 2-cell in $\text{Car}(\Lambda_1)\cap \text{Car}(\Lambda_2)$.
\end{lem}

\begin{proof}
Suppose, on the contrary, that there is another point of intersection. Then by Lemma \ref{Lemma:2collared_diagram} there is a 2-collared diagram with horns $\mathcal{D}$ collared by $\Lambda_1$ and $\Lambda_2$ and containing $\mathfrak{c}$. By Lemma \ref{Lemma:Shape_of_2_collared} we know that every 2-cell of $\mathcal{D}$ is distinguished, so this is a contradiction with the assumption that $\mathfrak{c}$ is regular.

Now suppose, on the contrary, that there is a 2-cell $\mathfrak{d} \in \text{Car}(\Lambda_1)\cap \text{Car}(\Lambda_2)$ such that $\mathfrak{d} \neq \mathfrak{c}$. We know that there is only one point of intersection of $\Lambda_1$ and $\Lambda_2$, so the only possibility is that $\mathfrak{d}$ is a distinguished 2-cell of some color, and $\Lambda_1$ and $\Lambda_2$ are both hypergraphs of this color (see Figure \ref{Figure:Regular_cell_intersection_once}).

\begin{figure}[h]
\centering
\begin{tikzpicture}[scale=0.5]
\draw  (-28,28) rectangle (-24,24);
\draw (-24,28) -- (-21,21) -- (-28,24);
\draw[thick] (-28,26) -- (-24,26) -- (-22.2996,24.0922);
\draw[thick] (-26,28) -- (-26,24) -- (-24.0922,22.3699);

\node at (-23,23) {$\mathfrak{d}$};
\node[above] at (-27,26) {$\Lambda_1$};
\node[right] at (-26,25) {$\Lambda_2$};
\node at (-25,27) {$\mathfrak{c}$};
\end{tikzpicture}
\caption{Two 2-cells in $\text{Car}(\Lambda_1)\cap \text{Car}(\Lambda_2)$.}
\label{Figure:Regular_cell_intersection_once}
\end{figure}
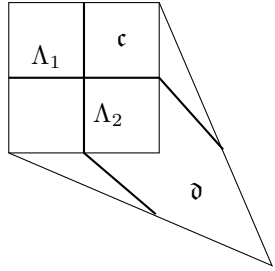 But then $\Lambda_1$ and $\Lambda_2$ intersect in the middle of the 2-cell strongly adjacent to $\mathfrak{d}$, so there are two points of intersection of $\Lambda_1$ and $\Lambda_2$, which is a contradiction.

\end{proof}

\begin{lem}\label{Lemma:Intersect_once_or_at_least_2_distinguished_cells}
Let $\gamma$ be a geodesic edge-path in $\cay^{(1)}$. Let $e$ be an edge in $\gamma$, and denote by $s$ its midpoint. Then either:
\begin{enumerate}[(a)]
\item There is a hypergraph intersecting $\gamma$ only in the point $s$.
\item There is a hypergraph $\Lambda$ such that its segment $\Lambda_{sx}$ starting in $s$ intersects $\gamma$ in two consecutive points: $s$ and $x$ such that there are at least two not sharing an edge distinguished 2-cells in $\Lambda_{sx}$ and $\Lambda$ does not intersect $\gamma$ anywhere else between points $s$ and $x$.
\end{enumerate}
\end{lem}

\begin{proof}
If one of the hypergraphs (red, blue or standard) dual to the edge $e$ intersects $\gamma$ only once, the statement is satisfied. Hence, consider the case where all hypergraphs dual to the edge $e$ intersect $\gamma$ at least twice. Denote by $x$, $y$ and $z$ the nearest points to $s$ (in the edge-path metric in $\cay^{(1)}$), in which three hypergraphs $\Lambda^{x}$, $\Lambda^{y}$ and $\Lambda^{z}$ dual to the edge $e$ intersect $\gamma$. A priori, it may happen that some of these points are equal. Without loss of generality, we can assume that $x$ and $y$ lie in the same connected component of $\gamma - s$ and that $x$ is at least as far from $s$ as $y$.  %By Lemma \ref{Lemma:There_is_a_color_cell_in_diagram} we know that the segment $\Lambda_{sx}$ of the hypergraph $\Lambda_{x}$ bounded by $s$ and $x$ contains a distinguished 2-cell $\mathfrak{c}_x$, which is red if $\Lambda_x$ is red, blue if $\Lambda_x$ is blue or $\Lambda_{sx}$ contains a pair of strongly adjacent 2-cells if $\Lambda_x$ is standard. Analogously the segment $\Lambda_{sy}$ of the hypergraph $\Lambda_y$ bounded by $s$ and $y$ contains a distinguished 2-cell $\mathfrak{c}_y$ (again of appropriate color) or a pair of strongly adjacent 2-cells.

Now, we will prove that $y \neq x$. Suppose, on the contrary, that $x=y$. This means that every time segments $\Lambda_{sx}$ and $\Lambda_{sy}$ split, they also rejoin at some point, since they end in the same point. Every situation where $\Lambda_{sx}$ and $\Lambda_{sy}$ split an rejoin gives us a 2-collared diagram collared by a subsegments of $\Lambda_{sx}$ and $\Lambda_{sy}$. By Lemma \ref{Lemma:Regular_Cell_Intersection_Once} every 2-collared diagram is a pair of strongly adjacent 2-cells. This means that if $\Lambda_{sx}$ and $\Lambda_{sy}$ split they rejoin in the next 2-cell of the common carrier, see Figure \ref{Figure:2collared_diagrams}. This means that $\text{Car}(\Lambda_{sx})=\text{Car}(\Lambda_{sy})$ and that distinguished 2-cells $\text{Car}(\Lambda_{sx})$ occur only in strongly adjacent pairs.

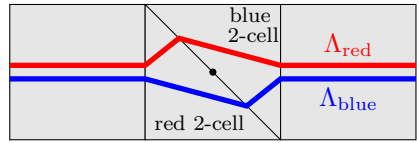
\begin{figure}[h]
\centering
\begin{tikzpicture}[scale=0.9]

%2 komorki:
\filldraw[draw=black,fill=gray!20] (-2,0) rectangle (0,2);
\filldraw[draw=black,fill=gray!20] (0,0) rectangle (2,2);
\filldraw[draw=black,fill=gray!20] (2,0) rectangle (4,2);
%\filldraw[draw=black,fill=gray!20] (0,2) rectangle (2,4);

%hipergraf

\path[black, draw] (0,2) -- (2,0);
\path[red, draw,line width=2pt] (-2,1.1) -- (0,1.1)--(0.5,1.5)--(2,1.1)--(4,1.1);
\path[blue, draw,line width=2pt] (-2,0.9) -- (0,0.9)--(1.5,0.5)--(2,0.9)--(4,0.9);

%wezly
%\draw (0.5,1.5) node[above left] {$\gamma'$};
\fill (1,1) circle (1.5pt);

\draw (3,1.1) node[above] { ${\color{red} \Lambda_{\text{red}}}$};
\draw (3,0.9) node[below] { ${\color{blue} \Lambda_{\text{blue}}}$};

\draw (0,0) node[above right] {\footnotesize{red 2-cell}};
\draw (1.55,2.1) node[below] {\footnotesize{blue}};
\draw (1.6,1.8) node[below] {\footnotesize{2-cell}};

\end{tikzpicture}
\caption{The only possible way of splitting and rejoining of $\Lambda_x$ and $\Lambda_y$.}
\label{Figure:2collared_diagrams}
\end{figure}

One of the hypergraphs $\Lambda_x$ and $\Lambda_y$ is red or blue, so by Lemma \ref{Lemma:There_is_a_color_cell_in_diagram} in the carrier of $\Lambda_{sx}$ exists an isolated distinguished 2-cell, so this is a contradiction. Therefore $y \neq x$.

Therefore, $y$ lies on a geodesic edge-path that intersects $\Lambda_{sx}$ twice. Hence, by Lemma \ref{Lemma:Geodesic_breaks_away_from_carrier} we know that $y$ lies in the boundary of the carrier of $\Lambda_{sx}$ (see Figure \ref{Figure:Intersect_once_or_at_least_2_distnguished_cells_Fig2} a) or $y$ lies the boundary of the house diagram $\mathcal{F}$ presented in Figure  \ref{Figure:Intersect_once_or_at_least_2_distnguished_cells_Fig2} b).
\begin{figure}[h]
\centering
\begin{tikzpicture}[scale=0.9]
%czesc a)

%2 komorki:
\filldraw[draw=black,fill=gray!20] (-5,-1) rectangle (-3,1);

%hipergraf

\path[red, draw,line width=2pt] (-5,0) -- (-3,0);
\draw (-3.5,-0.25) node { ${\color{red} \Lambda_x}$};

%geodezyjna
\draw[draw,line width=1.5pt, dashed] (-6,-0.5) -- (-5,-1) -- (-3,-1)--(-2,-0.5);

%krawedz
\path[black, draw,line width=2pt] (-5,-1) -- (-3,-1);

%wezly
\draw (-5, 1) node[below right] {$\mathfrak{c}$};
\fill (-4,-1) circle (2.5pt) node[below] {$y$};
\draw (-5.4,-0.75) node[below left] {$\gamma_{sx}$};
\draw (-4,-3) node {a)};

%czesc b)

%2 komorki:
\fill[fill=gray!20] (0,-1)--(2,-2.5)--(4,-1)--cycle;
\filldraw[draw=black,fill=gray!20] (0,-1) rectangle (2,1);
\filldraw[draw=black,fill=gray!20] (2,-1) rectangle (4,1);

%hipergraf
\path[red, draw,line width=2pt] (0,0) -- (4,0);
\draw (3.5,-0.25) node { ${\color{red} \Lambda_x}$};

\path[blue, draw,line width=2pt] (1,1) -- (1,-1) -- (3, -1.75);
\draw (1.3,0.7) node { ${\color{blue} \Lambda_y}$};

%geodezyjna
\draw[draw,line width=1.5pt, dashed] (-1,-1) -- (0,-1) -- (2,-2.5) -- (4,-1) -- (5,-1.5);

%krawedz
\path[black, draw,line width=2pt] (2,-2.5) -- (4,-1);

%wezly
\draw (0, 1) node[below right] {$\mathfrak{c}_1$};
\draw (2, 1) node[below right] {$\mathfrak{c}_2$};
\fill (3,-1.75) circle (2.5pt) node[below right] {$y$};
\draw (1,-1.75) node[below left] {$\gamma_{sx}$};
\draw (1,-3) node {b)};
\draw (3.6,-2.5) node {$\mathcal{F}$};

\end{tikzpicture}
\caption{Two possible locations of $y$.}
\label{Figure:Intersect_once_or_at_least_2_distnguished_cells_Fig2}
\end{figure}
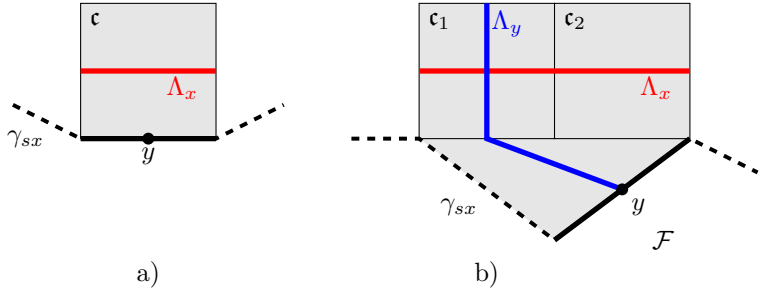

The hypergraphs $\Lambda_{x}$ and $\Lambda_{y}$ cannot intersect transversely in a regular 2-cell $\mathfrak{c}_1$, since otherwise by Lemma \ref{Lemma:Regular_Cell_Intersection_Once} they would intersect only once in the middle of $\mathfrak{c}_1$, and we know that they intersect also in $s$. Hence $y$ lies in the boundary of a distinguished 2-cell $\mathfrak{c}_y \in \text{Car}(\Lambda_{sx})$. Therefore the carriers of the segments $\Lambda_{sx}$ and $\Lambda_{sy}$ share a common distinguished 2-cell $\mathfrak{c}_y$.

Without loss of generality, we can suppose that $\mathfrak{c}_y$ is red (if it is blue the argument is analogous).

If $\Lambda^x$ or $\Lambda^y$ is blue then, according to Lemma \ref{Lemma:There_is_a_color_cell_in_diagram}, its carrier contains an isolated blue 2-cell $\mathfrak{c}_b$,  so $\mathfrak{c}_b$ does not share an edge with $\mathfrak{c}_y$. In that case the carrier of $\Lambda^x$ or $\Lambda^y$ contains a red 2-cell $\mathfrak{c}_y$ and some non-adjacent by an edge to it blue 2-cell $\mathfrak{c}_b$, so we are done.

The only cases left to consider are:
\begin{enumerate}
\item $\Lambda^y$ is a red hypergraph and $\Lambda^x$ is a standard hypergraph
\item $\Lambda^x$ is a red hypergraph and $\Lambda^y$ is a standard hypergraph.
\end{enumerate}

First consider the situation 1). If there are at least two red 2-cells in the carrier of $\Lambda_{sx}$, we are done. Hence, we can suppose that $\mathfrak{c}_y$ is the only isolated red 2-cell in the carrier of $\Lambda_{sx}$. In that case $\text{Car}(\Lambda_{sx})$ has the shape presented in Figure \ref{Figure:Shape_of_carriers1}. We know that $y$ lies  in the boundary of $\mathfrak{c}_y$ and $y$ does not belong to $\Lambda_{sx}$. Hence $y$ = $y_1$ or $y=y_2$. Recall that points $s, x$ and $y$ lie on $\gamma$. If $y=y_2$ then there is a segment of standard hypergraph bounded by $s$ and $y$ intersecting $\gamma$ twice at its ends, such that there is only one distinguished 2-cell in its carrier. This contradicts Lemma \ref{Lemma:There_is_a_color_cell_in_diagram} (3). If $y=y_2$ we obtain analogous contradiction. Hence there must be at least one more isolated red 2-cell in the carrier of $\Lambda_{sx}$, so we are done, since two red 2-cells cannot share an edge.

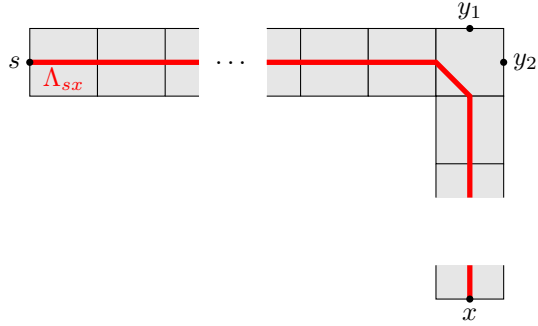
\begin{figure}[h]
\centering
\begin{tikzpicture}[scale=0.9]

%2 komorki:

\fill[fill=gray!20] (4,0) rectangle (4.5,1);
\fill[fill=gray!20] (5.5,0) rectangle (6,1);

\filldraw[draw=black,fill=gray!20] (2,0) rectangle (3,1);
\filldraw[draw=black,fill=gray!20] (3,0) rectangle (4,1);

\filldraw[draw=black,fill=gray!20] (6,0) rectangle (7,1);
\filldraw[draw=black,fill=gray!20] (7,0) rectangle (8,1);
\filldraw[draw=black,fill=gray!20] (8,0) rectangle (9,1);

\fill[fill=gray!20] (8,-1.5) rectangle (9,-1);
\fill[fill=gray!20] (8,-3) rectangle (9,-2.5);

\filldraw[draw=black,fill=gray!20] (8,-1) rectangle (9,0);

\draw (4,0)-- (4.5,0);
\draw (4,1)-- (4.5,1);
\draw (5.5,0)-- (6,0);
\draw (5.5,1)-- (6,1);

\draw (8,-1.5)-- (8,-1);
\draw (9,-1.5)-- (9,-1);
\draw (8,-2.5)-- (8,-3) -- (9,-3)--(9,-2.5);

%hipergraf

\path[red, draw,line width=2pt] (2,0.5) -- (4.5,0.5);
\path[red, draw,line width=2pt] (5.5,0.5) -- (8,0.5)--(8.5,0)--(8.5,-1.5);
\path[red, draw,line width=2pt] (8.5,-2.5)--(8.5,-3);

\draw (2.5,0.25) node { ${\color{red} \Lambda_{sx}}$};

\draw (5,0.5) node {$\dots$};

%wezly
\fill (8.5,1) circle (1.5pt) node[above] {$y_1$};
\fill (9,0.5) circle (1.5pt) node[right] {$y_2$};

\fill (2,0.5) circle (1.5pt) node[left] {$s$};
\fill (8.5,-3) circle (1.5pt) node[below] {$x$};

\end{tikzpicture}
\caption{The carrier of $\Lambda_{sx}$ in case where $\Lambda_x$ is red.}
\label{Figure:Shape_of_carriers1}
\end{figure}

\begin{figure}[h]
\centering
\begin{tikzpicture}[scale=0.9]

%2 komorki:

\fill[fill=gray!20] (4,0) rectangle (4.5,1);
\fill[fill=gray!20] (5.5,0) rectangle (6,1);

\filldraw[draw=black,fill=gray!20] (2,0) rectangle (3,1);
\filldraw[draw=black,fill=gray!20] (3,0) rectangle (4,1);

\filldraw[draw=black,fill=gray!20] (6,0) rectangle (7,1);
\filldraw[draw=black,fill=gray!20] (7,0) rectangle (8,1);
\filldraw[draw=black,fill=gray!20] (8,0) rectangle (9,1);

\fill[fill=gray!20] (8,-1.5) rectangle (9,-1);
\fill[fill=gray!20] (8,-3) rectangle (9,-2.5);

\filldraw[draw=black,fill=gray!20] (8,-1) rectangle (9,0);

\draw (4,0)-- (4.5,0);
\draw (4,1)-- (4.5,1);
\draw (5.5,0)-- (6,0);
\draw (5.5,1)-- (6,1);

\draw (8,-1.5)-- (8,-1);
\draw (9,-1.5)-- (9,-1);
\draw (8,-2.5)-- (8,-3) -- (9,-3)--(9,-2.5);

\draw (8,0)-- (9,1);

%hipergraf

\path[red, draw,line width=2pt] (2,0.5) -- (4.5,0.5);
\path[red, draw,line width=2pt] (5.5,0.5) -- (8,0.5)--(8.75,0.75)--(8.5,0)--(8.5,-1.5);
\path[red, draw,line width=2pt] (8.5,-2.5)--(8.5,-3);

\draw (2.5,0.25) node { ${\color{red} \Lambda_{sx}}$};

\draw (5,0.5) node {$\dots$};

%wezly
\fill (8.5,1) circle (1.5pt) node[above] {$y_3$};
\fill (9,0.5) circle (1.5pt) node[right] {$y_4$};

\fill (2,0.5) circle (1.5pt) node[left] {$s$};
\fill (8.5,-3) circle (1.5pt) node[below] {$x$};

\fill (8.5,0.5) circle (1.5pt);

\end{tikzpicture}
\caption{The carrier of $\Lambda_{sx}$ in case where $\Lambda_x$ is standard hypergraph.}
\label{Figure:Shape_of_carriers2}
\end{figure}
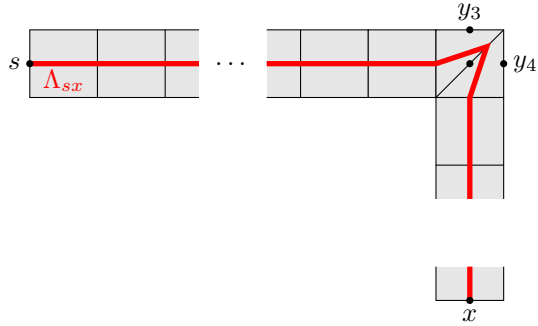

Now consider situation 2), that is $\Lambda^x$ is a standard hypergraph and there is only one divided tile of $\Lambda_{sx}$ and y belongs to this divided tile. In that case $\text{Car}(\Lambda_{sx})$ has the shape presented in Figure \ref{Figure:Shape_of_carriers2}. We know that $y$ lies in the boundary of $\mathfrak{c}_y$ but $y$ does not belong to $\Lambda_{sx}$. Moreover by \ref{Lemma:There_is_a_color_cell_in_diagram} part (3) we know that $y$ does not lie on an internal edge of a divided tile. Hence $y=y_3$ or $y=y_4$. If $y=y_3$ then there is a segment of a red hypegraph bounded by $x$ and $y$ intersecting $\gamma$ twice with no isolated red 2-cell in its carrier. This contradicts Lemma \ref{Lemma:There_is_a_color_cell_in_diagram} part (1). If $y=y_4$ we obtain the contradiction analogously. Hence, there must be at least two not sharing an edge 2-cells in the carrier of $\Lambda_{sx}$.

%The last case to consider is where $\Lambda_x$ and $\Lambda_y$ are standard and red hypergraph. If $\Lambda_{x}$ would be a red hypergraph

%(denote $\Lambda_s$) then by Lemma \ref{Lemma:There_is_a_color_cell_in_diagram} there is a pair of strongly adjacent 2-cells in the carrier of $\Lambda_s$. Apriori, it could happen that one of this 2-cells is $\mathfrak{c}_y$.

%Note that the hypergraph $\Lambda_y$ turns in a 2-cell $\mathfrak{c}_y$ since otherwise, there would be a transversal intersection of $\Lambda_x$ and $\Lambda_y$. This means, by the definition of colored hypergraphs, that the the 2-cell $\mathfrak{c}_y'$ strongly adjacent to $\mathfrak{c}_y$ is glued to it in a way, that $y$ lies on a distinguished edge. Hence by Lemma \ref{Lemma:There_is_a_color_cell_in_diagram} part (3) $\mathfrak{c}_y'$ there must be another pair of strongly adjacent 2-cells in its carrier. This gives us the statement.
\end{proof}

\begin{theo}\label{Theorem:Many_good_hypergraphs}
Let $\gamma$ be a geodesic edge-path in the 1-skeleton of $\cay$ of length at least 21. Let $E:=\{e_{-7}, e_{-6}, \dots, e_{0}, \dots, e_{7}\}$ be the set of 15 consecutive edges of $\gamma$. Then there is a hypergraph in $\cay$ intersecting $\gamma$ exactly once and such that this intersection is the midpoint of one of the edges in $E$.
\end{theo}

\begin{proof}
Denote the midpoint of the edge $e_0$ by $s$. If one of the hypergraphs dual to the edge $e_0$ intersects $\gamma$ only once, we are done. If not, by Lemma \ref{Lemma:There_is_a_color_cell_in_diagram} we know that there exists a hypergraph $\Lambda$, such that its segment $\Lambda_{sx}$ intersects $\gamma$ in two consecutive points of $\gamma$: $s$ and $x$ and that the segment $\Lambda_{sx}$ contains at least two not sharing an edge distinguished 2-cells: $\mathcal{d}_1$ and $\mathcal{d}_2$. We will consider two cases:

\begin{enumerate}
\item  There are at least seven 2-cells in $\text{Car}(\Lambda_{sx})$
\item  There are at most six 2-cells in $\text{Car}(\Lambda_{sx})$.
\end{enumerate}

We will use slightly different methods to deal with each of these two cases.

\subsection{Case 1: There are at least seven 2-cells in $\text{Car}(\Lambda_{sx})$}

Let $\mathfrak{c}_1, \mathfrak{c}_2, \mathfrak{c}_3$ be the 3-th, 4-th and 5-th 2-cells in $\Lambda_{sx}$ respectively counted from $s$. According to Lemma \ref{Lemma:Adjacent_to_special_cell} at least one of them must be a regular 2-cell, call it $\mathfrak{c}$. Denote by $\Lambda^{\mathfrak{c}}$ the standard hypergraph passing through 2-cell $\mathfrak{c}$ and dual to $\Lambda$. By Lemma \ref{Lemma:Regular_Cell_Intersection_Once} $\Lambda^{\mathfrak{c}}$ and $\Lambda$ intersect only once. We will prove two assertions:
\begin{enumerate}[(a)]
\item $\Lambda^{\mathfrak{c}}$ intersects some edge in $E$. \label{Point:a}
\item $\Lambda^{\mathfrak{c}}$ intersects $\gamma$ exactly once.
\label{Point:b}
\end{enumerate}

We start with the proof of \textbf{assertion (\ref{Point:a})}. Let us resolve where is the point of intersection of $\gamma$ and $\Lambda^{\mathfrak{c}}$. If there is an edge $e_{\mathfrak{c}}$ of $\gamma$ lying in the boundary of $\mathfrak{c}$ then $\Lambda^{\mathfrak{c}}$ and $\gamma$ intersect in the midpoint of $e_{\mathfrak{c}}$. In this case denote by $x_{\mathfrak{c}}$ the midpoint of $e_{\mathfrak{c}}$, see Figure \ref{Figure:Two_ways_of_intersection_of_Lambda_c_and_gamma} a). In that case there is an edge-path in $\cay^{(1)}$ that goes along the boundary of $\text{Car}(\Lambda)$ of length at most 6, because there are at most five 2-cells in $\text{Car}(\Lambda_{sx})$ between $s$ and $\mathfrak{c}$. Hence $d_{\gamma}(s,x_{\mathfrak{c}}) \leq 6$.

If there is no edge of $\gamma$ lying in the boundary of $\mathfrak{c}$ then, by Lemma \ref{Lemma:Geodesic_breaks_away_from_carrier} we know that the 2-cell $\mathfrak{c}$ lies in the house diagram presented in Figure \ref{Figure:Two_ways_of_intersection_of_Lambda_c_and_gamma} b) so $\Lambda_{\mathfrak{c}}$ intersects the edge $e_{\mathfrak{c}}$ from Figure \ref{Figure:Two_ways_of_intersection_of_Lambda_c_and_gamma} b). In that case we denote by the $x_{\mathfrak{c}}$ the midpoint of $e_{\mathfrak{c}}$. In this situation, there is an edge-path in $\cay^{(1)}$ that goes along the boundary of $\text{Car}(\Lambda)$ and a boundary of a house diagram containing $\mathfrak{c}$ of length at most 7, because there are at most five 2-cells in $\text{Car}(\Lambda_{sx})$ between $s$ and $\mathfrak{c}$. Hence $d_{\gamma}(s,x_{\mathfrak{c}}) \leq 7$.

This ends the proof of assertion (\ref{Point:a}).
\begin{figure}[h]
\centering
\begin{tikzpicture}[scale=0.9]
%czesc a)

%2 komorki:
\filldraw[draw=black,fill=gray!20] (-5,-1) rectangle (-3,1);

%hipergraf
\path[red, draw,line width=2pt] (-5,0) -- (-3,0);
\draw (-3.3,-0.25) node { ${\color{red} \Lambda}$};

\path[blue, draw,line width=2pt] (-4,1) -- (-4,-1);
\draw (-3.7,0.7) node { ${\color{blue} \Lambda^{\mathfrak{c}}}$};

%geodezyjna
\draw[draw,line width=1.5pt, dashed] (-6,-0.5) -- (-5,-1) -- (-3,-1)--(-2,-0.5);

%krawedz
\path[black, draw,line width=2pt] (-5,-1) -- (-3,-1);

%wezly
\draw (-5, 1) node[below right] {$\mathfrak{c}$};
\draw (-4.7,-1) node[below] {$e_{\mathfrak{c}}$};
\fill (-4,-1) circle (2.5pt) node[below] {$x_{\mathfrak{c}}$};
\draw (-5.4,-0.75) node[below left] {$\gamma$};
\draw (-4,-3) node {a)};

%czesc b)

%2 komorki:
\fill[fill=gray!20] (0,-1)--(2,-2.5)--(4,-1)--cycle;
\filldraw[draw=black,fill=gray!20] (0,-1) rectangle (2,1);
\filldraw[draw=black,fill=gray!20] (2,-1) rectangle (4,1);

%hipergraf
\path[red, draw,line width=2pt] (0,0) -- (4,0);
\draw (3.5,-0.25) node { ${\color{red} \Lambda}$};

\path[blue, draw,line width=2pt] (1,1) -- (1,-1) -- (3, -1.75);
\draw (1.3,0.7) node { ${\color{blue} \Lambda^{\mathfrak{c}}}$};

%geodezyjna
\draw[draw,line width=1.5pt, dashed] (-1,-1) -- (0,-1) -- (2,-2.5) -- (4,-1) -- (5,-1.5);

%krawedz
\path[black, draw,line width=2pt] (2,-2.5) -- (4,-1);

%wezly
\draw (0, 1) node[below right] {$\mathfrak{c}$};
\draw (2.2,-2.25) node[below right] {$e_{\mathfrak{c}}$};
\fill (3,-1.75) circle (2.5pt) node[below right] {$x_{\mathfrak{c}}$};
\draw (1,-1.75) node[below left] {$\gamma$};
\draw (1,-3) node {b)};

\end{tikzpicture}
 \caption{Two possible ways of intersection of $\Lambda_{\mathfrak{c}}$ and $\gamma$}
 \label{Figure:Two_ways_of_intersection_of_Lambda_c_and_gamma}
\end{figure}
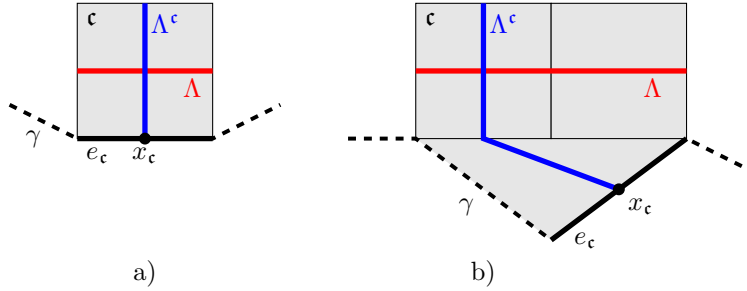

Now we will prove \textbf{assertion (\ref{Point:b})}. By Lemma \ref{Lemma:Connected_components} there are exactly two connected components of $\cay - \Lambda$. Hence, we can split $\Lambda^{\mathfrak{c}}$ into two parts:
\begin{itemize}
\item $\Lambda_+$ which is the
 intersection of $\Lambda^{\mathfrak{c}}$ with the component of $\cay - \Lambda$ containing $x_{\mathfrak{c}}$
\item $\Lambda_-$ which is the intersection of $\Lambda^{\mathfrak{c}}$ with the component of $\cay - \Lambda$ not containing $x_{\mathfrak{c}}$.
\end{itemize}

We will show that $\Lambda_-$ does not intersect $\gamma$ and that $\Lambda_+$ intersects $\gamma$ only once (in the point $x_{\mathfrak{c}}$).

Firstly, consider $\Lambda_+$. Suppose, on the contrary, that $\Lambda_{+}$ intersects $\gamma$ in some point $y_\mathfrak{c} \neq x_{\mathfrak{c}}$. Suppose  that $y_{\mathfrak{c}}$ is the nearest among what point to $x_{\mathfrak{c}}$ in the metric $d_{\gamma}$. Points $x_c$ and $y_c$ cannot both lie in the boundary of $\mathfrak{c}$ since there is at most one common edge of $\mathfrak{c}$ and $\gamma$.

We will consider three cases:

\begin{enumerate}
\item $y_{\mathfrak{c}}$ belongs to the segment $\gamma_{sx}$ of $\gamma$ bounded by $s$ and $x$, \label{case_1}, see Figure \ref{Figure:Cases} (1)
\item $y_{\mathfrak{c}}$ does not belong to $\gamma_{sx}$ and points $x_{\mathfrak{c}}$ and $y_{\mathfrak{c}}$ lie in the different connected components of $\cay - \Lambda$, \label{case_2}, see Figure \ref{Figure:Cases} (2)
\item $y_{\mathfrak{c}}$ does not belong to $\gamma_{sx}$ and points $x_{\mathfrak{c}}$ and $y_{\mathfrak{c}}$ lie in the same connected component of $\cay - \Lambda$. \label{case_3}, see Figure \ref{Figure:Cases} (3)
\end{enumerate}

%%%%%%%%%%%%%%%%%%%%%%%%%%%%%%
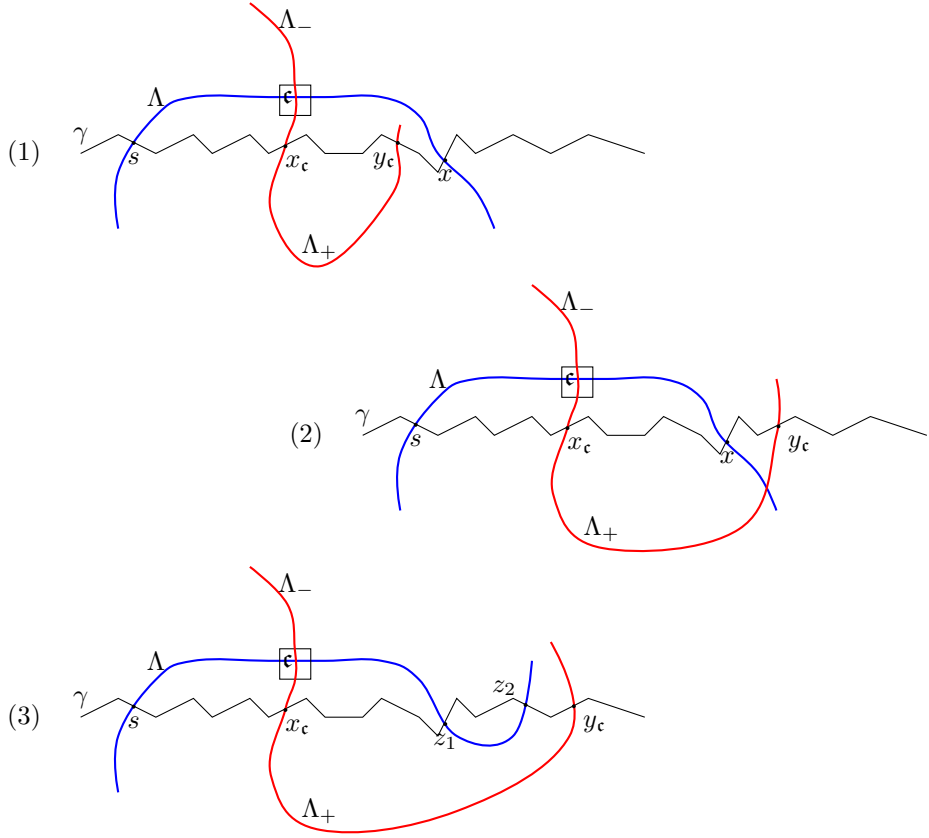
\begin{figure}[h!]
\centering
\begin{tikzpicture}[scale=0.25]
%% Case 1:

\draw (-32,22) -- (-30,23) -- (-28,22) -- (-26,23) -- (-25,22) -- (-23,23) -- (-22,22) -- (-20,23) -- (-19,22) -- (-17,22) -- (-16,23) -- (-14,22) -- (-13,21) -- (-12,23) -- (-11,22) -- (-9,23) -- (-7,22) -- (-5,23) -- (-2,22);
\draw [blue, thick] plot[smooth, tension=.7] coordinates {(-30,18) (-30,21) (-28,24) (-26,25) (-22,25) (-19,25) (-16,25) (-14,24) (-13,22) (-11,20) (-10,18)};
\draw  (-21.4142,25.6401) rectangle (-19.7622,24.0232);
\draw [red, thick]  plot[smooth, tension=.7] coordinates {(-23,30) (-21,28) (-20.6104,25.625) (-20.5906,24.0179) (-21.1263,22.3711) (-21.8604,18.8196) (-19.2613,16.0022) (-15.432,19.8315) (-15.174,22.5695) (-14.9955,23.5219)};

% podpisy
\node at (-19.3,17) {$\Lambda_+$};
\node at (-28,25) {$\Lambda$};
\node at (-21,25) {$\mathfrak{c}$};

\node[below right] at (-21.6106,22.384) {$x_\mathfrak{c}$};
\fill (-21.1106,22.384) circle (3pt);

\node[below left] at (-14.7352,22.5628) {$y_\mathfrak{c}$};
\fill (-15.1352,22.5628) circle (3pt);

\node[below] at (-29.1811,22.5308) {$s$};
\fill (-29.1811,22.5608) circle (3pt);

\node[below] at (-12.6109,21.6361) {$x$};
\fill (-12.6109,21.6361) circle (3pt);

\node[above] at (-32,22) {$\gamma$};
\node[right] at (-22,29) {$\Lambda_-$};

\node at (-35,22) {(1)};

%% Case 2:

\begin{scope}[shift={(15,-15)}]

\draw (-32,22) -- (-30,23) -- (-28,22) -- (-26,23) -- (-25,22) -- (-23,23) -- (-22,22) -- (-20,23) -- (-19,22) -- (-17,22) -- (-16,23) -- (-14,22) -- (-13,21) -- (-12,23) -- (-11,22) -- (-9,23) -- (-7,22) -- (-5,23) -- (-2,22);
\draw [blue, thick]  plot[smooth, tension=.7] coordinates {(-30,18) (-30,21) (-28,24) (-26,25) (-22,25) (-19,25) (-16,25) (-14,24) (-13,22) (-11,20) (-10,18)};
\draw  (-21.4142,25.6401) rectangle (-19.7622,24.0232);
\draw [red, thick] plot[smooth, tension=.7] coordinates {(-23,30) (-21,28) (-20.6104,25.625) (-20.5906,24.0179) (-21.1263,22.3711) (-21.8604,18.8196) (-19.2613,16.0022) (-12,17) (-10,22) (-10,25)};
[blue, draw,line width=2pt]

% podpisy
\node at (-19.3,17) {$\Lambda_+$};
\node at (-28,25) {$\Lambda$};
\node at (-21,25) {$\mathfrak{c}$};

\node[below right] at (-21.6106,22.384) {$x_\mathfrak{c}$};
\fill (-21.1106,22.384) circle (3pt);

\node[below right] at (-9.8898,22.4767) {$y_\mathfrak{c}$};
\fill (-9.8898,22.4767) circle (3pt);

\node[below] at (-29.1811,22.5308) {$s$};
\fill (-29.1811,22.5608) circle (3pt);

\node[below] at (-12.6109,21.6361) {$x$};
\fill (-12.6109,21.6361) circle (3pt);

\node[above] at (-32,22) {$\gamma$};
\node[right] at (-22,29) {$\Lambda_-$};

\node at (-35,22) {(2)};
\end{scope}

%% Case 3:
\begin{scope}[shift={(0,-30)}]

\draw (-32,22) -- (-30,23) -- (-28,22) -- (-26,23) -- (-25,22) -- (-23,23) -- (-22,22) -- (-20,23) -- (-19,22) -- (-17,22) -- (-16,23) -- (-14,22) -- (-13,21) -- (-12,23) -- (-11,22) -- (-9,23) -- (-7,22) -- (-5,23) -- (-2,22);
\draw [blue, thick] plot[smooth, tension=.7] coordinates {(-30,18) (-30,21) (-28,24) (-26,25) (-22,25) (-19,25) (-16,25) (-14,24) (-12,21) (-9,21) (-8,25)};
\draw  (-21.4142,25.6401) rectangle (-19.7622,24.0232);
\draw  [red, thick] plot[smooth, tension=.7] coordinates {(-23,30) (-21,28) (-20.6104,25.625) (-20.5906,24.0179) (-21.1263,22.3711) (-21.8604,18.8196) (-19.2613,16.0022) (-12,17) (-6,21) (-7,26)};

% podpisy
\node at (-19.3,17) {$\Lambda_+$};
\node at (-28,25) {$\Lambda$};
\node at (-21,25) {$\mathfrak{c}$};

\node[below right] at (-21.6106,22.384) {$x_\mathfrak{c}$};
\fill (-21.1106,22.384) circle (3pt);

\node[below right] at (-5.7727,22.5821) {$y_\mathfrak{c}$};
\fill (-5.7727,22.5821) circle (3pt);

\node[below] at (-29.1811,22.5308) {$s$};
\fill (-29.1811,22.5608) circle (3pt);

\node[below] at (-12.6109,21.6361) {$z_1$};
\fill (-12.6109,21.6361) circle (3pt);

\node[above left] at (-8.3386,22.6524) {$z_2$};
\fill (-8.3386,22.6524) circle (3pt);

\node[above] at (-32,22) {$\gamma$};
\node[right] at (-22,29) {$\Lambda_-$};

\node at (-35,22) {(3)};

\end{scope} 
\end{tikzpicture}
\caption{Intersections of  $\Lambda_+$ and $\gamma$.}
\label{Figure:Cases}
\end{figure}
%%%%%%%%%%%%%%%%%%%%%%%%%%%%%%%

\ref{case_1})If $y_{\mathfrak{c}}$ belongs to the segment $\gamma_{sx}$ of $\gamma$ bounded by $s$ and $x$. Then by Lemma \ref{Lemma:Intersecting_intersecting_geodesic} applied to $\Lambda_1 = \Lambda$ and $\Lambda_2 = \Lambda^{\mathfrak{c}}$ we obtain that there are two common 2-cells of the carriers of $\Lambda$ and $\Lambda_{\mathfrak{c}}$, which is a contradiction with Lemma \ref{Lemma:Regular_Cell_Intersection_Once}.

\ref{case_2}) If points $x_{\mathfrak{c}}$ and $y_{\mathfrak{c}}$ lie in the different connected components of $\cay - \Lambda$ then there is an intersection of $\Lambda_+$ which contradicts the fact that $\Lambda$ and $\Lambda^{\mathfrak{c}}$ intersect only once ( Lemma \ref{Lemma:Regular_Cell_Intersection_Once}).

\ref{case_3}) If points $x_c$ and $y_c$ lie in the same connected component of $\cay - \Lambda$ then $\Lambda$ intersects $\gamma_{sx}$ at least twice in some points $z_1, z_2$.

If this two points do not lie in the boundary of the same 2-cell $\mathfrak{c}'$ in $\text{Car}(\Lambda_{\mathfrak{c}})$ then $\mathfrak{c}' \neq \mathfrak{c}$. By Lemma \ref{Lemma:Intersecting_intersecting_geodesic} applied for $\Lambda_1=\Lambda_{\mathfrak{c}}$ and  $\Lambda_2 = \Lambda$ we obtain that there are two different 2-cells in $\text{Car}(\Lambda) \cap \text{Car}(\Lambda_{\mathfrak{c}})$, which contradicts Lemma \ref{Lemma:Regular_Cell_Intersection_Once}. If $z_1$ and $z_2$ lie in the boundary of the same 2-cell $\mathfrak{c}'$ in $\text{Car}(\Lambda_{\mathfrak{c}})$ then we know, that $\mathfrak{c}' \neq \mathfrak{c}$, since $\mathfrak{c}$ shares at most one edge with $\gamma$. Hence, $\mathfrak{c}', \mathfrak{c} \in \text{Car}(\Lambda_{\mathfrak{c}}) \cap \text{Car}(\Lambda)$, which again contardicts \ref{Lemma:Regular_Cell_Intersection_Once}. Hence, we proved that $\Lambda_{+}$ does not intersect $\gamma$. 

Let us now consider $\Lambda_-$. Suppose, on the contrary, that $\Lambda_-$ intersects $\gamma$. Let $t_{\mathfrak{c}}$ be the point of the intersection of $\Lambda_-$ and $\gamma$, which is the nearest to the point $x_{\mathfrak{c}}$ in the edge-path metric on $\Lambda_-$ (see Figure \ref{Figure:Intersection_of_gamma_minus}). Note, that $t_{\mathfrak{c}}$ lies in the connected component of $\cay - \Lambda$ not containing $x_{\mathfrak{c}}$,  since otherwise there will be a second intersection of $\Lambda$ and $\Lambda_{\mathfrak{c}}$ which is not allowed, according to Lemma \ref{Lemma:Regular_Cell_Intersection_Once}.

\begin{figure}[h]
\centering
\begin{tikzpicture}[scale=0.4]

\draw (-41.0675,21.2264) -- (-39.0776,22.7458) -- (-36.0776,21.7458) -- (-33.9998,23.0964) -- (-32,22) -- (-30.2325,23.0414) -- (-28,22) -- (-26,23) -- (-25,22) -- (-23,23) -- (-22,22) -- (-20,23) -- (-19,22) -- (-17,22) -- (-16,23) -- (-14,22) -- (-13,21) -- (-12,23) -- (-11,22) ;
\draw [blue, thick] plot[smooth, tension=.7] coordinates {(-30,18) (-30,21) (-28,24) (-26,25) (-22,25) (-19,25) (-16,25) (-14,24) (-13,22) (-11,20) (-10,18)};
\draw  (-21.4142,25.6401) rectangle (-19.7622,24.0232);
\draw [red, thick]  plot[smooth, tension=.7] coordinates {(-23,30) (-21,28) (-20.6104,25.625) (-20.5906,24.0179) (-21.1263,22.3711) (-21.8604,18.8196)    };

% podpisy
\node[right] at (-21.1027,20.1078) {$\Lambda_+$};
\node at (-28,25) {$\Lambda$};
\node at (-21,25) {$\mathfrak{c}$};

\node[below right] at (-21.6106,22.384) {$x_\mathfrak{c}$};
\fill (-21.1106,22.384) circle (3pt);

\node[below] at (-29.1811,22.5308) {$s$};
\fill (-29.1811,22.5608) circle (3pt);

\node[right] at (-12.6109,21.6361) {$x$};
\fill (-12.6109,21.6361) circle (3pt);

\node[above] (v1) at (-32,22) {$\gamma$};
\node[ right] at (-27.89,28.2456) {$\Lambda_-$};

\draw[red, thick]  plot[smooth, tension=.7] coordinates {(-23,30) (-24.8426,30.6455) (-26.678,29.2409) (-29.3444,27.6569) (-32.4611,27.0336) (-35.4738,24.9212) (-37.2745,22.8435) (-38.0018,20.3848) (-38.2788,18.2378)};
\node[below right] at (-37.5169,22.218) {$t_\mathfrak{c}$};

\fill (-37.5169,22.2618) circle (3pt);

\end{tikzpicture}
 \caption{Intersection of $\Lambda_{-}$ and $\gamma$}
 \label{Figure:Intersection_of_gamma_minus}
\end{figure}
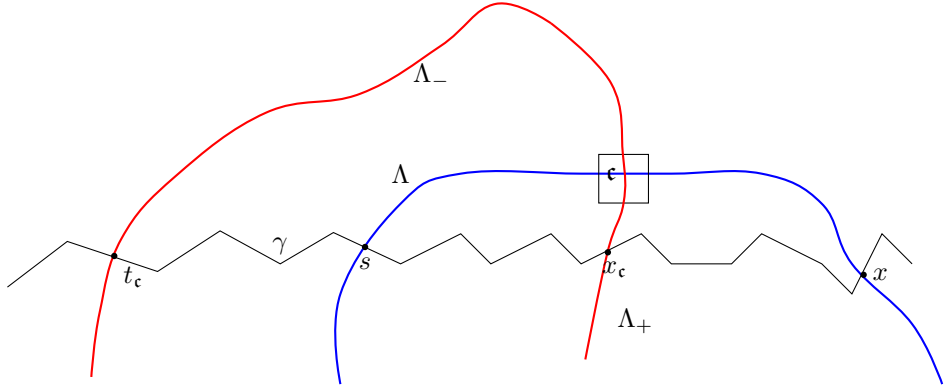

Let $\Lambda'_-$ be the hypergraph segment of $\Lambda_{\mathfrak{c}}$ bounded by the points $x_{\mathfrak{c}}$ and $t_{\mathfrak{c}}$. One of the points $\{x,s\}$  lies on $\gamma$ between points $x_{\mathfrak{c}}$ and $t_{\mathfrak{c}}$. Suppose first that this point is $s$.

%Note, that $x$ is the only one point of intersection of $\Lambda_{\mathfrak{c}}$, since otherwise by Lemma \ref{Lemma:Intersecting_intersecting_geodesic} hypergraphs $\Lambda$ and $\Lambda_{\mathfrak{c}}$ would intersect twice.
Since $s$ lies on the geodesic segment, that intersects $\Lambda^{\mathfrak{c}}$ twice, by Lemma \ref{Lemma:Geodesic_breaks_away_from_carrier} we have two possibilities: $s$ lies on the boundary of the 2-cell $\mathfrak{c}$ (see Figure \ref{Figure:Two_ways_of_locations_of_x} a) ) or $s$ lies in the diagram presented in Figure \ref{Figure:Two_ways_of_locations_of_x} b).

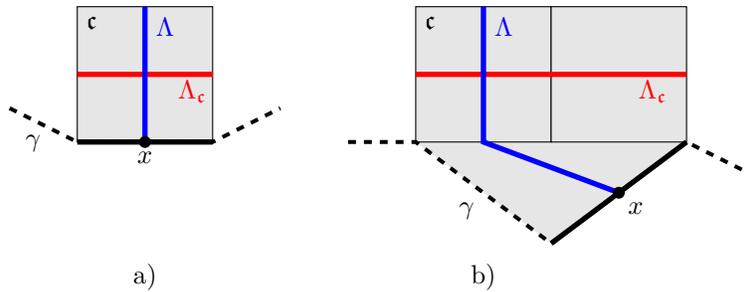
\begin{figure}[h]
\centering
\begin{tikzpicture}[scale=0.9]
%czesc a)

%2 komorki:
\filldraw[draw=black,fill=gray!20] (-5,-1) rectangle (-3,1);

%hipergraf
\path[red, draw,line width=2pt] (-5,0) -- (-3,0);
\draw (-3.3,-0.25) node { ${\color{red} \Lambda_{\mathfrak{c}}}$};

\path[blue, draw,line width=2pt] (-4,1) -- (-4,-1);
\draw (-3.7,0.7) node { ${\color{blue} \Lambda}$};

%geodezyjna
\draw[draw,line width=1.5pt, dashed] (-6,-0.5) -- (-5,-1) -- (-3,-1)--(-2,-0.5);

%krawedz
\path[black, draw,line width=2pt] (-5,-1) -- (-3,-1);

%wezly
\draw (-5, 1) node[below right] {$\mathfrak{c}$};
\fill (-4,-1) circle (2.5pt) node[below] {$x$};
\draw (-5.4,-0.75) node[below left] {$\gamma$};
\draw (-4,-3) node {a)};

%czesc b)

%2 komorki:
\fill[fill=gray!20] (0,-1)--(2,-2.5)--(4,-1)--cycle;
\filldraw[draw=black,fill=gray!20] (0,-1) rectangle (2,1);
\filldraw[draw=black,fill=gray!20] (2,-1) rectangle (4,1);

%hipergraf
\path[red, draw,line width=2pt] (0,0) -- (4,0);
\draw (3.5,-0.25) node { ${\color{red} \Lambda_{\mathfrak{c}}}$};

\path[blue, draw,line width=2pt] (1,1) -- (1,-1) -- (3, -1.75);
\draw (1.3,0.7) node { ${\color{blue} \Lambda}$};

%geodezyjna
\draw[draw,line width=1.5pt, dashed] (-1,-1) -- (0,-1) -- (2,-2.5) -- (4,-1) -- (5,-1.5);

%krawedz
\path[black, draw,line width=2pt] (2,-2.5) -- (4,-1);

%wezly
\draw (0, 1) node[below right] {$\mathfrak{c}$};
\fill (3,-1.75) circle (2.5pt) node[below right] {$x$};
\draw (1,-1.75) node[below left] {$\gamma$};
\draw (1,-3) node {b)};

\end{tikzpicture}
 \caption{Two possible locations of $x$}
 \label{Figure:Two_ways_of_locations_of_x}
\end{figure}

Note, that in both cases the segment of $\Lambda$ joining the point $s$ and the 2-cell $\mathfrak{c}$ consists of at most one 2-cell, which, if occurs, is a regular 2-cell. Remember that between $s$ and $x$ in $\text{Car}(\Lambda)$ was at least seven 2-cells and $\mathfrak{c}$ is 3-rd, 4-th or 5-th of them. Hence there must be a least two 2-cells between $\mathfrak{c}$ and $s$. This is a contradiction, so  $\Lambda_-$ does not intersect $\gamma$.

If $x$ lies on $\gamma$ between $x_{\mathfrak{c}}$ and $t_{\mathfrak{c}}$ the reasoning is analogous.   This ends the proof in Case 1.

\subsection{Case 2: There are at most six 2-cells in $\text{Car}(\Lambda_{sx})$}

Now let us consider the case, where there are at most six 2-cells in the carrier of $\Lambda_{sx}$. Let $\mathfrak{c}$ be any regular 2-cell in the carrier of $\Lambda_{sx}$ that lies between $\mathcal{d}_1$ and $\mathcal{d}_2$ (recall that $\mathcal{d}_1$ and $\mathcal{d}_2$ were two not sharing an edge distinguished 2-cells in the segment $\Lambda_{sx}$). Again, we denote by $\Lambda_{\mathfrak{c}}$ the hypergraph passing through $\mathfrak{c}$ and dual to $\Lambda_{sx}$. Analogously we state two assertions:

\begin{enumerate}[(a)]
\item $\Lambda^{\mathfrak{c}}$ intersects some edge of the set $E$. \label{Point2:a}
\item $\Lambda^{\mathfrak{c}}$ intersects $\gamma$ exactly once.
\label{Point2:b}
\end{enumerate}

To prove assertions (\ref{Point2:a}) and (\ref{Point2:b}) we first define the point $x_{\mathfrak{c}}$ as in Case 1.

Proof of assertion (\ref{Point2:a}) is analogous as in Case 1.

Now, consider assertion (\ref{Point2:b}). Again, we split $\Lambda^{\mathfrak{c}}$ into two connected components:

\begin{itemize}
\item $\Lambda_+$ which is the
 intersection of $\Lambda^{\mathfrak{c}}$ with the component of $\cay - \Lambda$ containing $x_{\mathfrak{c}}$
\item $\Lambda_-$ which is the intersection of $\Lambda^{\mathfrak{c}}$ with the component of $\cay - \Lambda$ not containing $x_{\mathfrak{c}}$.
\end{itemize}

To prove that $\Lambda_+$ intersects $\gamma$ only once we can repeat the argument from Case 1.

Consider now $\Lambda_-$. Suppose, on the contrary that $\Lambda_-$ intersects $\gamma$ in a point $t_{\mathfrak{c}}$. Note that $t_{\mathfrak{c}}$ lies in the connected component of $\cay - \Lambda$ not containing $x_{\mathfrak{c}}$, so the geodesic segment of $\gamma$ bounded by $x_\mathfrak{c}$ and $t_{\mathfrak{c}}$ contains $x$ or $s$. Without loss of generality we can assume that this geodesic segment contains $x$.

Let $\Lambda'_-$ be the hypergraph segment of $\Lambda_{\mathfrak{c}}$ bounded by the points $x_{\mathfrak{c}}$ and $t_{\mathfrak{c}}$. The same reasoning, as in previous case, shows that there is at most one 2-cell in $\text{Car}(\Lambda)$ between $\mathfrak{c}$ and $x$, which, if occurs, is a regular 2-cell. This is a contraction since $\mathfrak{c}$ lies in $\text{Car}(\Lambda)$ between two distinguished 2-cells, so there must be a distinguished 2-cell in $\text{Car}(\Lambda)$ between $\mathfrak{c}$ and $x$.

\end{proof}

Thorem \ref{Theorem:Proper_action} is an immediate consequence of Theorem \ref{Theorem:Many_good_hypergraphs}.

\end{document}